\documentclass[11pt]{amsart}
\usepackage{lscape}
\usepackage[foot]{amsaddr}

\usepackage{ifpdf}
\usepackage{setspace}
\usepackage[font=footnotesize,parskip=5pt,width=\textwidth]{caption,subfig}
\ifpdf

\newcommand{\drivy}{pdftex}
\usepackage[pdftex=true,hyperfootnotes=true]{hyperref}
\else

\newcommand{\drivy}{dvips}
\usepackage[hyperfootnotes=true]{hyperref}
\fi

\hypersetup{colorlinks=true,raiselinks=true,breaklinks=false}

\usepackage[\drivy]{graphicx}
\usepackage{natbib}
\usepackage{rccol}
\usepackage{ctable}
\definecolor{orange}{rgb}{1,0.5,0}

\numberwithin{equation}{section}

\newtheorem{thm}{Theorem}[section]
\newtheorem{cor}[thm]{Corollary}
\newtheorem{lem}[thm]{Lemma}

\theoremstyle{definition}

\newtheorem{remark}[thm]{Remark}
\newtheorem{assumption}{Assumption}
\newtheorem{example}{Example}[section]
\newtheorem{algo}[thm]{Algorithm}

\theoremstyle{remark}

\newcommand{\norm}[1]{\left\Vert#1\right\Vert}
\newcommand{\abs}[1]{\left\vert#1\right\vert}

\newcommand{\Real}{\mathbb R}
\newcommand{\Rplus}{\mathbb {R}_+}

\newcommand{\moment}{\mu}

\DeclareMathOperator{\diag}{diag}

\newcommand{\brac}[1]{\left ( #1 \right )}
\newcommand{\brak}[1]{\left [ #1 \right ]}
\newcommand{\brat}[1]{\left \{ #1 \right \}}
\newcommand{\Qmeas}{\mathbb{Q}}

\newcommand{\Pmeas}{\mathbb{P}}
\newcommand{\LS}{\mathcal{L}}

\newcommand{\ev}[1]{\mathbb{E}\brak{#1}}
\newcommand{\evj}{\mathbb{E}^{(J)}}
\newcommand{\vr}[1]{\mathbb{V}\brak{#1}}
\newcommand{\evt}[2]{\mathbb{E}_{#1}\brak{#2}}
\newcommand{\evtj}[2]{\mathbb{E}^{(J)}_{#1}\brak{#2}}

\newcommand{\bracket}[1]{\left \langle #1 \right \rangle }
\newcommand{\domain}{\mathcal{D}}

\newcommand{\reals}{\mathbb{R}}

\newcommand{\e}{\ensuremath{\mathrm{e\;\!}}}
\newcommand{\iim}{\ensuremath{\mathrm{i}}} 
\usepackage[latin1]{inputenc}

\usepackage{dcolumn}
\newcolumntype{d}{D{.}{.}{3}}

\title[Density Expansions]{Density Approximations for Multivariate Affine Jump-Diffusion Processes}
\author{Damir Filipovi\'{c}$^1$}
\address{$^1$\'Ecole Polytechnique F\'ed\'erale de Lausanne and Swiss Finance Institute, Quartier UNIL-Dorigny, Extranef 218, CH - 1015 Lausanne, Switzerland}
\email{\href{mailto:damir.filipovic@epfl.ch}{\texttt{damir.filipovic@epfl.ch}}}
\author{Eberhard Mayerhofer$^2$}
\address{$^2$Vienna Institute of Finance, Heiligenst\"{a}dter Str. 46-48, 1190 Vienna, Austria}
\email{\href{mailto:eberhard.mayerhofer@vif.ac.at}{\texttt{eberhard.mayerhofer@vif.ac.at}}}
\author{Paul Schneider$^3$}
\address{$^3$Warwick Business School, University of
Warwick, Coventry CV4 7AL, United Kingdom}
\email{\href{mailto:paul.schneider@wbs.ac.uk}{\texttt{paul.schneider@wbs.ac.uk}}}
\thanks{We are thankful to Yacine A\"{i}t-Sahalia, Michael Brandt, Anna Cieslak, Pierre Collin-Dufresne, Valentina Corradi, Ron Gallant, Aleksandar Mijatovi\'{c}, Alessandro Palandri, Benedikt Pötscher, and Gareth Roberts for helpful discussions. We benefitted from suggestions from participants of the Workshop on Financial Econometrics at the Fields institute, Toronto, the internal workshop at Warwick Business School, Coventry, the Econometric Research Seminar at the IHS, Vienna, and the 2010 meeting of the European Finance Association, Frankfurt. Part of this research has been carried out within the project on "Dynamic Asset Pricing" of the National Centre of Competence in Research "Financial Valuation and Risk Management" (NCCR FINRISK). The NCCR FINRISK is a research instrument of the Swiss National Science Foundation. Eberhard Mayerhofer gratefully acknowledges support from WWTF (Vienna Science and Technology Fund).}

\date{13 October 2011}

\begin{document}
%
%
%
%

\begin{abstract}

We introduce closed-form transition density expansions for
multivariate affine jump-diffusion processes.  The expansions
rely on a general approximation theory which we develop in weighted Hilbert spaces
for random variables which possess all polynomial moments. We
establish parametric conditions which guarantee existence and differentiability of transition densities of affine models and
show how they naturally fit into the approximation framework.
Empirical applications in credit risk, likelihood inference, and
option pricing highlight the usefulness of our expansions. The
approximations are extremely fast to evaluate, and they perform very
accurately and numerically stable.

\end{abstract}
\maketitle


\section{Introduction}

Most observed phenomena in financial markets are inherently
multivariate: stochastic trends, stochastic volatility, and the
leverage effect in equity markets are well-known examples. The
theory of affine processes provides multivariate stochastic models
with a well established theoretical basis and sufficient degree of
tractability to model such empirical attributes. They enjoy much
attention and are widely used in practice and academia. Among their
best-known proponents are Vasicek's interest rate model
\citep[][]{vasicek77}, the square-root model \citet{cir85}, Heston's
model \citep[cf.][]{Heston1993}, and affine term structure models
\citep{duffiekan96,daisingleton00,collindufresnegoldsteinjones06}.
Affine models owe their popularity and their name to their key
defining property: their characteristic function is of exponential
affine form and can be computed by solving a system of generalized
Riccati differential equations (cf.\ \citet{duffiefilipovicschachermayer03}). This allows for computing
transition densities and transition probabilities by means of
Fourier inversion \citep{duffiepansingleton00}. Transition densities
constitute the likelihood which is an  ingredient for both
frequentist and Bayesian econometric methodologies.\footnote{Various
other approaches for parameter estimation for discretely observed
Markov processes can be found in the literature \citep[excellent
comprehensive surveys are for example
in][]{hurnjeismanlindsay05,sorensen,aitsahalia06}.  The approaches
range from likelihood approximation using Bayesian data augmentation
\citep{robertsstramer01,elerianchibshephard01, eraker01,jones98},
estimating functions \citep{bibbyjacobsensorensen04},  up to the
efficient method of moment \citet{gallanttauchen02}. Only few of
them make use of the properties of affine models, however
\citep[e.g.][]{singleton01,bates2005}.} Also, they appear in the
pricing of financial derivatives. However, Fourier inversion is a
very delicate task. Complexity and numerical difficulties increase
with the dimensionality of the process. Efficient density
approximations avoiding the need for Fourier inversion are therefore
desirable.

This  paper is concerned with directly approximating the transition
density  without resorting
to Fourier inversion techniques. We pursue a polynomial
expansion approach, an idea that has been proposed by \citet{schoutens00},
\citet{aitsahalia02} and \citet{hurnjeismanlindsay08} among others for univariate
diffusion processes. Extensions for multivariate (jump-)diffusions
do exist in \citet{aitsahalia08} and \citet{yu07}, but they follow a
different route by approximating the Kolmogorov forward-, and
backward partial differential equations.  Our approach  exploits
a crucial property of affine processes. Under some technical
conditions, conditional moments of all orders exist and are explicitly given in terms of derivatives of the affine characteristic exponential function, see \citet{duffiepansingleton00}. This
ensures that the coefficients of the polynomial expansions can be
computed without approximation error.


We present a general theory of density approximations with several traits of the affine model class in mind.  The assumptions made for the general theory are then justified by proving existence and differentiability  of the true, unknown transition densities of affine models. These theoretical results, contrary to the density approximations themselves, do rely on Fourier theory. Specifically we investigate the asymptotic behavior of the characteristic function with novel ODE techniques.

We improve
earlier work, along
several lines. Our method (i) is
applicable to multivariate models; (ii) works equally well for
reducible and irreducible processes in the sense of
\citet{aitsahalia08}\footnote{A model is said to be \emph{reducible}
\citet{aitsahalia08} if its diffusion function can be transformed
one-to-one into a constant)}, in particular stochastic volatility
models; (iii) produces density approximations the quality of which
is independent of the time interval between observations; (iv)
allows for expansions on the ''correct" state space. That is, the
support of the density approximation agrees with the support of the
true, unknown transition density as in \citet{hurnjeismanlindsay08} and \citet{schoutens00};
(v) produces density approximations that integrate to unity by
construction, hence are much more amenable to applications that demand the constant of proportionality than the
purely polynomial expansions from \citet{aitsahalia08}.\footnote{The Markov chain Monte Carlo sampling schemes from
\citet{stramerbognarschneider09} accommodate Bayesian
likelihood-based inference using expansions from
\citet{aitsahalia08} even in absence of the normalizing constant,
but at a high computational cost.}  A specialization on affine models is not a severe limitation, since virtually any continuous-time multivariate application is based on affine models.\footnote{In discrete-time, \citet{lesingletondai10} show how $\mathbb{Q}$-affine models may be constructed to exhibit non-affine dynamics under $\Pmeas$.} This includes
Wishart processes \citet{bru91} and even general affine
matrix-valued processes \citep{cuchieroetal09}. This paper therefore provides a unified framework for  econometric inference for financial models, because in applications one typically needs to evaluate, both, the transition densities themselves, as well as integrals of payoff functions against the transition densities for model-based asset pricing. This complements the methods recently developed in \citet{chenjoslin11} and \citet{kristensenmele11}, which are aimed at asset pricing only.\footnote{It is of course conceivable to mix the mentioned methods. For example, one could use transition densities developed in this paper, while approximating asset prices using the generalized Fourier transform in \citet{chenjoslin11}, whenever the payoff function allows it, or the error expansion method from \citet{kristensenmele11}.}

The paper proceeds as follows: Section
\ref{sec:densityapproximations} develops a general theory of orthonormal polynomial density approximations in certain weighted $\LS ^{2}$
spaces--under suitable integrability and regularity assumptions.
These may be validated by the sufficient criteria presented
subsequently in Section \ref{secsuffass}. The density approximations are
then specialized within the context of affine processes: Section
\ref{sec:affinemodels} reviews the affine transform
formula and the polynomial moment formula for affine processes, which in
turn allows the aforementioned polynomial approximations. The main
theoretical contribution--constituted by fairly general results on
existence and differentiability of transition densities of affine
processes-- is elaborated in Section \ref{sec:truedensities}. In
Section \ref{sec:weightfunctions} we introduce candidate weight
functions and the Gram-Schmidt algorithm to compute orthonormal polynomial bases
corresponding to these weights, along with important examples.
Section \ref{sec:relationexistingapproximations} relates
existing techniques for density approximations to ours. An empirical
study is presented in Section \ref{sec:applications}:
applications in stochastic volatility (Section \ref{sec:hestonmodel}), credit risk (Section \ref{sec:cdopricing}),
likelihood inference (Section \ref{sec:likelihoodinference}), and
option pricing (Section \ref{sec:optionpricing}), support the
tractability and usefulness of the  likelihood expansions. Section
\ref{sec:discussion} concludes. The proofs of our main results are given in Appendices \ref{approofslemma}--\ref{proof: int cbi dens}.

In the paper we will use the following notational conventions. The nonnegative integers are denoted by $\mathbb N_0$. The length of a multi-index $\alpha=(\alpha_1,\dots,\alpha_d)\in\mathbb N_0^d$ is defined by $|\alpha| =\alpha_1+\cdots+\alpha_d$, and we write $\xi^\alpha=\xi_1^{\alpha_1} \cdots  \xi_d^{\alpha_d}$ for any $\xi\in\mathbb R^d$. The degree of a polynomial $p(x)=\sum_{|\alpha|\ge 0} p_\alpha\,x^\alpha$ in $x\in \mathbb R^d$ is defined as $\deg p(x) =\max\{|\alpha|\mid p_\alpha\neq 0\}$. For the likelihood ratio functions below we define $0/0=0$. The class of $p$-times continuously differentiable (or continuous, if $p=0$) functions on $\mathbb R^d$ is denoted by $C^p$.

\section{Density Approximations}\label{sec:densityapproximations}

Let $g$ denote a probability density on $\mathbb R^d$ whose polynomial moments
\[ \moment _\alpha =\int_{\mathbb R^d} \xi^\alpha\,g(\xi)\,d\xi\]
of every order $\alpha\in \mathbb N_0^d$ exist and are known in closed form. For example, $g$ may denote the pricing density in a financial market model. Typically, $g$ is not known in explicit form, and needs to be approximated. Let $w$ be an auxiliary probability density function on $\mathbb R^d$. The aim is to expand the likelihood ratio $g/w$ in terms of orthonormal polynomials of $w$ in order to get an explicit approximation for the unknown density function $g$. This can be formalized as follows. Define the weighted Hilbert space $\LS ^2_{w}$ as the set of (equivalence classes of) measurable functions $f$ on $\mathbb R^d$ with finite $\LS ^2_{w}$-norm defined by
\[ \| f\|_{\LS^2_w} ^2 = \int_{\mathbb R^d} |f(\xi)|^2\,w(\xi)\,d\xi<\infty .\]
Accordingly, the scalar product on $\LS^2_w$ is denoted by
\[ \langle f,h\rangle_{\LS^2_w} = \int_{\mathbb R^d} f(\xi)\,h(\xi)\,w(\xi)\,d\xi.
\]
We will now proceed under the following assumptions. Sufficient conditions for the assumptions to hold are provided in Section~\ref{secsuffass} below.
\begin{assumption}\label{ass1}
 There exists an orthonormal basis of polynomials $\{H_\alpha\mid\alpha\in\mathbb N_0^d\}$ of $\LS^2_w$ with $\deg H_\alpha = |\alpha|$. This implies $H_0=1$ in particular.
\end{assumption}

\begin{assumption}\label{ass2}
  The likelihood ratio function $g/w$ lies in $\LS^2_w$. This is equivalent to $\int_{\mathbb R^d} \frac{g(\xi)^2}{w(\xi)}\,d\xi<\infty$.
\end{assumption}
Consequently, the coefficients
\[ c_\alpha = \left\langle\frac{g}{w},H_\alpha  \right\rangle_{\LS^2_w} = \int_{\mathbb R^d} H_\alpha(\xi) \,g(\xi)\,d\xi \quad\text{($=1$ for $\alpha=0$)}\]
are well defined and given explicitly\footnote{This is an advantage over the method in \citet{aitsahalia02} which also relies on series expansions, where the coefficients are functions of expectations of nonlinear moments, and therefore have to be approximated in general.} in terms of the coefficients of $H_\alpha$ and the polynomial moments $\moment _\alpha$ of $g$. Moreover, according to standard $\LS ^2_{w}$-theory, the sequence of pseudo-likelihood ratios\footnote{See Footnote \ref{f1} below for an explanation of this terminology.} $1+\sum_{|\alpha|=1}^J c_\alpha H_\alpha$ approximates the likelihood ratio $g/w$ in $\LS ^2_{w}$ for $J\to\infty$. In fact, defining the pseudo-density functions\footnote{\label{f1}Theorem~\ref{thmgJ} below states that $g^{(J)}$ integrates to one, but $g^{(J)}$ may take negative values. Whence we shall call $g^{(J)}$ a pseudo-density function, and $g^{(J)}/w$ a pseudo-likelihood ratio.}
\begin{equation}\label{eq:truncatedexpansion}
  g^{(J)}(x) = w(x) \left(1+\sum_{|\alpha|=1}^J c_\alpha H_\alpha(x)\right)
  \end{equation}
the following properties can be established.
\begin{thm}\label{thmgJ}
The pseudo-density functions $g^{(J)}$ satisfy
\begin{gather}
  \int_{\mathbb R^d}g^{(J)}(\xi)\,d\xi =1\label{lemgJ1}\\
  \lim_{J\to \infty} \frac{g^{(J)}}{w}=\frac{g}{w}\quad \text{in $\LS ^2_{w}$}\label{lemgJ2}\\
  \lim_{J\to \infty}\int_{\mathbb R^d}\left| g^{(J)}(\xi)-g(\xi)\right|^2 \frac{d\xi}{w(\xi)}=0\label{lemgJ3}.
\end{gather}
\end{thm}
Property \eqref{lemgJ1} proves to be very useful for applications where the constant of proportionality is needed, for example option pricing and the computation of Bayes factors.
\begin{proof}
A calculation shows that
\[ \int_{\mathbb R^d}H_\alpha(\xi)\,w(\xi)\,d\xi = \bracket{H_\alpha, 1}_{\LS _w^{2}}=\bracket{H_\alpha, H_0}_{\LS _w^{2}}=0,\]
by the orthogonality of $H_\alpha$ and $H_0=1$. Hence $\int_{\mathbb R^d}  g^{(J)}(\xi)\,d\xi =    \int_{\mathbb R^d} w(\xi)\,d\xi =1$, which proves \eqref{lemgJ1}. Properties \eqref{lemgJ2} and \eqref{lemgJ3} are formal restatements of the discussion preceding the theorem.
\end{proof}

The idea of expanding the likelihood ratio function $g/w$ in orthonormal polynomials of $w$ is simple and powerful. An overview and discussion of related literature can be found e.g.\ in \citet{ber_95}. In particular, for the case where $w$ is the standard Gaussian density, \eqref{eq:truncatedexpansion} is actually the Gram--Charlier expansion of $g$. But note that Assumption~\ref{ass2} is very restrictive in this case. This is why the Gram--Charlier series diverges in most cases of interest, which is sometimes given as an argument against the use of it. However, the blame is on the choice of the Gaussian as auxiliary density.
The efficiency of the approximation \eqref{lemgJ2}, or equivalently \eqref{lemgJ3}, lies in the appropriate choice of the auxiliary density function $w$ and the corresponding orthonormal polynomials $H_\alpha$.

Here is a first result towards a good choice of $w$. The intuition is to choose $w$ as close as possible to the unknown density function
$g$, in the sense that the pseudo-likelihood ratio $g/w$ is close to one. This should be achieved if many of the coefficients $c_\alpha$, other than $c_0=1$, are equal to zero. This will also improve the numerical efficiency of the approximation as the respective orthonormal polynomials $H_\alpha$ need not be computed. Denote the polynomial moments of $w$ by
\[
\lambda_\alpha=\int_{\mathbb R^d} \xi^\alpha\, w(\xi)\,d\xi.
\]
\begin{lem}[Moment Matching Principle]\label{thm:momentmatching}
   Suppose for some $n\geq 1$, we have $\mu_\alpha=\lambda_\alpha$ for all $|\alpha|\leq n$. Then $c_\alpha = 0$ for $1\leq |\alpha|\leq n$.
  \end{lem}
\begin{proof}
The assumption implies that, for $1\le |\alpha|\le n$,
\[ c_{\alpha} = \int_{\mathbb R^d} H_\alpha(\xi) \,g(\xi)\,d\xi = \int_{\mathbb R^d} H_\alpha(\xi) \,w(\xi)\,d\xi = \bracket{H_\alpha, 1}_{\LS _w^{2}}=\bracket{H_\alpha, H_0}_{\LS _w^{2}}=0,\]
by the orthogonality of $H_\alpha$ and $H_0=1$.
\end{proof}

\section{Sufficient Conditions for Assumptions~\ref{ass1} and \ref{ass2}}\label{secsuffass}
In this section we provide sufficient conditions for Assumptions~\ref{ass1} and \ref{ass2} to hold. The proofs of the following lemmas are postponed to Appendix~\ref{approofslemma}. We first provide sufficient conditions on $w$ that guarantee that Assumption~\ref{ass1} is satisfied.
\begin{lem}\label{lempolydense}
Suppose that the density function $w$ has a finite exponential moment
\begin{equation}\label{eps0mom}
 \int_{\mathbb R^d}\e^{\epsilon_0 \|\xi\|}\,w(\xi)\,d\xi<\infty
\end{equation}
for some $\epsilon_0>0$. Then the set of polynomials is dense in $\LS^2_w$. Moreover, Assumption~\ref{ass1} is satisfied.
\end{lem}

In applications, the auxiliary density function $w$ on $\mathbb R^d$ will often be given as product of marginal densities $w_i$ on $\mathbb R$. Hence the following modification of Lemma~\ref{lempolydense} will be useful.
\begin{lem}\label{lemprod}
Let $w_1,\dots,w_d$ be density functions on $\mathbb R$ having
finite exponential moments
\[ \int_{\mathbb R} \e^{\epsilon_i  |\xi_i |}\,w_i(\xi_i)\,d\xi_i<\infty  \]
for some $\epsilon_i>0$, $i=1,\dots, d$.
Then the product density $w(\xi)=w_1(\xi_1)\cdots w_d(\xi_d)$ on
$\mathbb R^d$ admits a finite exponential moment \eqref{eps0mom} for $\epsilon_0=\min_i\epsilon_i$. Moreover, let
$\{H^i_{j}\mid j\in\mathbb N_0\}$ denote the corresponding
orthonormal basis of polynomials of $\LS^2_{w_i}(\mathbb R)$ by
$\deg H^i_{j}=j$, for $i=1,\dots,d$, asserted by
Lemma~\ref{lempolydense}. Then
\[ H_\alpha(\xi) = H^1_{\alpha_1}(\xi_1)\cdots H^d_{\alpha_d}(\xi_d) \]
defines an orthonormal basis of polynomials of $\LS^2_w$ with $\deg H_\alpha= |\alpha|$, and Assumption~\ref{ass1} is satisfied.
\end{lem}

Assumption~\ref{ass2} is opposite to Assumption~\ref{ass1} in the sense that there we have to bound the auxiliary density function $w$ from below.
The following lemmas provide sufficient conditions for Assumption~\ref{ass2} to hold.
\begin{lem}\label{lemass2simple}
Assume that $g$ is bounded and has a finite exponential moment
\[ \int_{\mathbb R^d}\e^{\epsilon_0 \|\xi\|}\,g(\xi)\,d\xi<\infty \]
for some $\epsilon_0>0$. If $w$ decays at most exponentially such that
\begin{equation}\label{lemass2simplecond}
  \sup_{x\in\mathbb R^d} \frac{\e^{-\epsilon_0 \|x\|}}{w(x)}<\infty
\end{equation}
then Assumption~\ref{ass2} is satisfied.
\end{lem}

If the support of $w$ and $g$ is contained in a subset $\domain$ of $\mathbb R^d$, the situation becomes more difficult as one has to control the rate at which $w$ converges to zero at the boundary of the support set. We provide sufficient conditions for the set $\domain=\mathbb R^m_+\times \mathbb R^n$, starting with the scalar case $\domain=\mathbb R_+$.
\begin{lem}\label{lemass2spec}
Let $d=1$ and $p\in\mathbb N$. Assume that $g$ is a bounded density with support in $\mathbb R_+$ and has a finite exponential moment
\[ \int_0^\infty \e^{\epsilon_0 \xi}\,g(\xi)\,d\xi<\infty \]
for some $\epsilon_0>0$. Assume further that $g$ is of class $C^p$. If $w$ has support in $\mathbb R_+$, and decays at most polynomially at zero and exponentially at infinity such that
\begin{equation}\label{lemass2specass}
 \sup_{x\in [0,1]} \frac{  x^{2p} }{w(x)}<\infty  \quad\text{and}\quad  \sup_{x\ge 1} \frac{ \e^{-\epsilon_0 x} }{w(x)}<\infty
\end{equation}
then Assumption~\ref{ass2} is satisfied.
\end{lem}

The case where $\domain=\mathbb R_+^m\times \mathbb R^n$ is similar, but requires stronger conditions on $g$ and $w$.
We respect the product structure of the domain by writing $g=g(x,y)$ for $x\in\mathbb R_+^m$ and $y\in\mathbb R^n$.
The following tubular neighborhood of the boundary of $\domain$
\[ \mathcal I = \mathbb R_+^m\setminus (1,\infty)^m =\{ x\in \mathbb R_+^m\mid \min_i x_i\le 1\}\]
is the convenient multivariate generalization of the unit interval from the above scalar case.
\begin{lem}\label{lemass2}
Let $d=m+n$ and $p\in\mathbb N$. Assume that $g(x,y)$ is a bounded density with support in $\mathbb R_+^m\times \mathbb R^n$ and
has a finite exponential moment
\[ \int_{\mathbb R^m_+} \int_{\mathbb R^n} \e^{\epsilon_1 \|\xi\| + \epsilon_2 \| \eta\|}\,g(\xi,\eta)\,d\xi\,d\eta<\infty \]
for some $\epsilon_1,\,\epsilon_2>0$. Assume further that $g(x,y)$ is of class $C^p$ in $x$ and the
$p$-th partial derivative $\partial_{x_i}^p g(x,y)$ is bounded on $\mathcal I\times\mathbb R^n$,
for all $i=1,\dots,m$. If $w$ has support in $\mathbb R^m_+\times \mathbb R^n$,
and decays at most polynomially around the boundary and exponentially at infinity such that
\begin{equation}\label{lemass2ass}
 \sup_{(x,y)\in \mathcal I\times\mathbb R^n}
 \frac{  \min_i x_i^{p}\, \e^{  - \epsilon_2\| y\|}  }{w(x,y)}<\infty  \quad\text{and}\quad  \sup_{(x,y)\in (1,\infty)^m\times \mathbb R^n} \frac{ \e^{-\epsilon_1 \|x\| - \epsilon_2\| y\|} }{w(x,y)}<\infty
\end{equation}
then Assumption~\ref{ass2} is satisfied.
\end{lem}

We note that the conditions in Lemmas~\ref{lemass2simple}, \ref{lemass2spec} and \ref{lemass2} can be explicitly verified for transition densities of affine processes, see Corollary~\ref{coraffineC} below.

\section{Affine Models}\label{sec:affinemodels}
The main application of the polynomial density approximation is for affine factor models. In this section, we follow the setup of \citet{duffiefilipovicschachermayer03}, which we now briefly recap. Let $d=m+n\ge 1$. We define the index set $J=\{m+1,\dots,d\}$, and write $v_J=(v_{m+1},\dots,v_d)$ and $m_{JJ}=\left(m_{kl}\right)_{k,l\in J}$, for any vector $v$ and matrix $m$. We consider an affine process $X$ on the canonical state space $\domain=\mathbb R_+^m\times\mathbb R^n$ with generator
\begin{align*}
 \mathcal A f(x)&= \sum_{k,l=1}^d  \left(\diag\left(0, a\right)+\sum_{i=1}^m x_i\alpha_{i}\right)_{kl}\,\frac{\partial^2 f(x)}{\partial x_k\partial x_l} + \left(b+\beta\,x\right)^\top\nabla  f(x)\\
 &\quad+ \int_\domain \left( f(x+\xi)-f(x)-\chi_J(\xi)^\top\nabla_J f(x)\right) m(d\xi)\\
 &\quad+ \int_\domain \left( f(x+\xi)-f(x)\right)\left(\sum_{i=1}^m x_i\mu_i(d\xi)\right)
\end{align*}
for some appropriate positive semidefinite $n\times n$- and $d\times d$-matrices $a$ and $\alpha_i$, respectively. Here, with $\diag\left(0, a\right)$ we denote the block-diagonal $d\times d$-matrix with blocks given by the $m\times m$-zero matrix and $a$. Moreover, $\chi_J(\xi)$ denotes an $\mathbb R^n$-valued continuous and bounded truncation function with $\chi_J(\xi)=\xi_J$ in a neighborhood of the origin $\xi=0$. For detailed parametric
restrictions on $(a,\alpha_i, b,\beta,m,\mu_i)$ we refer the reader to \citet[Definition
2.6]{duffiefilipovicschachermayer03}. We assume for simplicity\footnote{At
the cost of more technical analysis, the following results could also
be proved for the general case of infinite variation jumps $\mu_i$
with infinite tail mean.} that the jump measures $\mu_i$ are of
finite variation type with integrable large jumps
\[\int_{\domain}  \|\xi\| \, \mu_i(d\xi)<\infty,\quad i=1,\dots,m.\]

\subsection{Affine Transform Formula}
The analytical tractability of affine models stems from the fact that the characteristic function of $X_t|X_0=x$ is explicitly given by the affine
transform formula
\begin{equation}\label{eqatfXX}
  \mathbb E\left[ e^{\iim u^\top X_t} \mid X_0=x\right]=e^{\phi(t,\iim u)+\psi(t,\iim u)^\top x} ,\quad u\in\mathbb R^d,\quad x\in \domain
\end{equation}
where the $\mathbb C_-$- and $\mathbb C^m_-\times\iim\mathbb
R^n$-valued functions $\phi=\phi(t,\iim  u)$ and $\psi=\psi(t,\iim u)$ solve the generalized Riccati equations, for $i=1,\dots,m$,
\begin{equation}\label{ricceq}
  \begin{aligned}
    \partial_t\phi&=\psi_J^\top a\,\psi_J+b^\top\psi+\int_{\domain}\left(e^{\psi^\top\xi}-1-\psi_J^\top\chi_J(\xi)\right)m(d\xi),\\
    \phi(0)&=0,\\
        \partial_t\psi_i&=\psi^\top\alpha_i\,\psi+\mathcal B_i\,\psi+\int_{\domain}\left(e^{\psi^\top\xi}-1\right)\mu_i(d\xi),\\
        \psi_i(0)&=\iim u_i, \\
    \partial_t\psi_J&=\mathcal B_{JJ}\psi_J,\\
    \psi_J(0)&=\iim u_J,
  \end{aligned}
\end{equation}
where we define $\mathcal
B=\beta^\top$ and write $\mathcal B_i$ for the $i$th row vector of
$\mathcal B$.
Obviously, we have
\[ \psi_J(t,\iim u)=\iim e^{\mathcal B_{JJ} t }\,u_J,\]
and $\phi(t,\iim u)$ is given by simple integration of the right hand
side of its equation.

\subsection{Polynomial Moments}
It is well known that if $X_t|X_0=x$ has finite $k$-th moment,
\[\mathbb E\left[ \|X_t\|^k \mid X_0=x\right]<\infty\quad \text{for all $x\in\domain$,}\]
then $\phi(t,u)$ and $\psi(t,u)$ are of class $C^k$ in $u$. Moreover, the polynomial moments are explicitly given in terms of the respective mixed derivatives of the characteristic function
\[ \mathbb E\left[  X_t^\alpha \mid X_0=x\right]=-\iim^{|\alpha|}\frac{\partial^{|\alpha|}}{\partial u_{\alpha_1}\cdots\partial u_{\alpha_d}}\, e^{\phi(t,\iim u)+\psi(t,\iim u)^\top x}|_{u=0} \]
for $|\alpha|\le k$, see e.g.\ \citet[Lemma A.1]{duffiefilipovicschachermayer03}. It follows by inspection that the right hand side of this equation is a real polynomial in $x$ of degree less than or equal to $|\alpha|$. Recently, generalizing the recursive method used in \citet{formansorensen08} for Pearson-type diffusions, \citet{cuchieroetal08} proposed an alternative method to compute the coefficients of this polynomial. The idea rests on the insight that the affine generator $\mathcal A$ formally maps $\mathcal P_k$ into $\mathcal P_k$, where $\mathcal P_k$ denotes the finite-dimensional linear space of all polynomials in $x\in\mathbb R^d$ of degree less than or equal to $k$.\footnote{This method is not restricted to affine processes, but can be defined for any Markov process with finite $k$-th moments, and whose infinitesimal generator maps $\mathcal P_k$ into itself.} The generator $\mathcal A$ thus restricts to a linear operator $\mathcal A_k$ on $\mathcal P_k$. Consequently, we obtain the formal representation
\[\mathbb E\left[  X_t^\alpha \mid X_0=x\right] = \e^{\mathcal A_k  t}\, x^\alpha\]
where $\e^{\mathcal A_k  t}=\sum_{j\ge 0} \frac{\left(\mathcal A_k t\right)^j}{j!}$ is the exponential of $\mathcal A_k t$. This can be expressed as a matrix. We shall illustrate this for $d=1$. The dimension of $\mathcal P_k$ then equals $k+1$, and we can pick as canonical basis of $\mathcal P_k$ the set $\mathcal Q=\{ 1,x,\dots,x^k\}$. For every $j=0,\dots,k$ we then calculate symbolically the coefficients $q_{ij}$ in
\begin{equation}\label{eqQdef}
  \mathcal A_k x^j = \mathcal A x^j =\sum_{i=0}^k q_{ij} x^i .
\end{equation}
Hence $\mathcal A_k$ can be represented by the upper-triangular matrix $Q=\left(q_{ij}\right)$ with respect to the basis $\mathcal Q$. In other words, if we identify a generic polynomial $p(x)=\sum_{i=0}^k p_i x^i$ in $\mathcal P_k$ with the vector of its coefficients $p=(p_0,\dots,p_k)^\top$, then $\mathcal A_k p(x)\in\mathcal P_k$ equals the polynomial with coefficient vector $Qp$. Moreover,
\begin{equation}\label{eq:momentformula}
\mathbb E\left[  p(X_t) \mid X_0=x\right] = \e^{ Q  t}\, p.
\end{equation}
See Examples~\ref{example:laguerre} and \ref{example:heston} below for some concrete applications for $d=1$ and $d=2$.

\subsection{Existence and Properties of Affine Transition Densities} \label{sec:truedensities}
In this section we present our main theoretical results, which establish existence and smoothness of the density of the conditional distribution of the affine process $X$. Moreover, we provide explicitly verifiable conditions asserting that Lemmas~\ref{lemass2simple}--\ref{lemass2} apply.

Our first result provides sufficient conditions for the existence and smoothness of
a density of $X_t|X_0=x$, for some $t>0$ and $x\in\mathbb R^m_+\times\mathbb R^n$. These easy to check conditions apply in particular to multi-factor affine term-structure models on $\mathbb R^m_+ \times \mathbb R^n$ from \citet{duffiekan96} and \citet{daisingleton00}, and Heston's stochastic volatility model. The proof is given in Appendix~\ref{sec_proofthmtailnew}.

\begin{thm}\label{thmtailnew}
Assume that the $d\times (n+1)d$-matrix\footnote{Here, for given $d\times d$-matrices $B_1,\,B_2,\,\dots,\,B_n$ the expression $\left[ B_1,\,B_2,\,\dots,\,B_n\right]$ denotes the $d\times nd$-block matrix we obtain by putting the matrices next to each other.}
\begin{equation}\label{Kcaldef}
 \mathcal K=\left[ \sum_{i=1}^m\alpha_i ,\, \diag\left(0,a\right) ,\, \diag\left(0,\mathcal B_{JJ}^\top a\right),\dots,\, \diag\left(0,(\mathcal B_{JJ}^{n-1})^\top a\right)\right]
\end{equation}
has full rank. Further, let $p$ be a nonnegative integer with
\begin{equation}\label{eq:jakobsebastian}
p< \min_{i\in \{1,\dots,m\}} \frac{b_i}{\alpha_{i,ii}}-1.
\end{equation}
Then $X_t|X_0=x$ admits a density $g(\xi)$ of class $C^p$ with support in $\mathbb R^m_+\times\mathbb R^n$ and the partial
derivatives of $g(\xi)$ of orders $0,\dots,p$ tend to $0$ as
$\|\xi\|\to\infty$.
\end{thm}

We note that condition \eqref{eq:jakobsebastian} is sharp and cannot be relaxed in general. Consider for instance the scalar square-root diffusion $X$ on $\mathbb R_+$ with generator $\mathcal A f(x)=\alpha x f''(x)+ b f'(x)$. It is well known that for any parameter values $\alpha>0$ and $b\ge 0$, the distribution of $\frac{2 X_t}{\alpha t}\mid X_0=x$ is noncentral $\chi^2$ with $\frac{2b}{\alpha}$ degrees of freedom and noncentrality parameter $\frac{2 x}{\alpha t}$, see e.g.\ \citet[Exercise 10.9]{filipovicbook2009}. The corresponding density function $g(\xi)$ satisfies $\lim_{\xi\to 0} g(\xi)=0$, and is therefore of class $C^0$, if and only if the degrees of freedom $\frac{2b}{\alpha}>2$, see \citet[Chap.\ 29]{johkotbal95}. This is exactly what condition \eqref{eq:jakobsebastian} states for $p=0$.

As regards exponential moments of $X_t|X_0=x$, we combine and
rephrase some results from \citet{duffiefilipovicschachermayer03}:
\begin{thm}\label{thmmoments}
Assume that the jump measures admit exponential moments
\[ \int_{\{\|\xi\|>1\}} e^{q^\top\xi}\,m(d\xi)<\infty\quad\text{and}\quad \int_{\{\|\xi\|>1\}} e^{q^\top\xi}\,\mu_i(d\xi)<\infty,\quad i=1,\dots,m\]
for all $q$ in some open neighborhood $V$ of $0$ in $\mathbb R^d$. Then the right hand side of \eqref{ricceq} is analytic in $\psi\in
V$. Suppose further that \eqref{ricceq} admits a $V$-valued solution $\psi(t,u)$ with $\psi(0,u)=u$ for all $t\in [0,T)$ and for all
$u$ in $[-\epsilon_1,\epsilon_1]^m\times [-\epsilon_2,\epsilon_2]^n$, for some $\epsilon_1,\epsilon_2>0$. Then
$X_t|X_0=x$ has a finite exponential moment
\begin{equation}\label{eqexpqXt}
\mathbb E\left[\e^{\epsilon_1 \|Y_t\|+\epsilon_2 \|Z_t\|}\mid X_0=x\right]<\infty
\end{equation}
for all $t\in [0,T)$, where we denote $Y_t=(X_{1,t},\dots,X_{m,t})^\top$ and $Z_t=(X_{m+1,t},\dots,X_{d,t})^\top$.
\end{thm}

\begin{proof}
That the right hand side of \eqref{ricceq} is analytic in $\psi\in
V$ follows from \citet[Lemma 5.3]{duffiefilipovicschachermayer03}. From \citet[Theorem 2.16 and Lemma 6.5]{duffiefilipovicschachermayer03} we then infer that
\[  \mathbb E\left[\e^{q^\top X_t}\mid X_0=x\right]<\infty \]
for all $t\in [0,T)$ and for all $q\in [-\epsilon_1,\epsilon_1]^m\times [-\epsilon_2,\epsilon_2]^n$. Combining this with the elementary inequality
\[ \e^{\epsilon_1 \|Y_t\|+\epsilon_2 \|Z_t\|}\le \e^{\epsilon_1 \sum_{i=1}^m |X_{it}|+\epsilon_2 \sum_{i=m+1}^d |X_{it}|}\le \sum_{|\alpha|=0}^1\e^{ q_\alpha^\top X_{t}},\]
where we denote $q_\alpha=\left((-1)^{\alpha_1}\epsilon_1,\dots,(-1)^{\alpha_m}\epsilon_1,
(-1)^{\alpha_{m+1}}\epsilon_2,\dots,(-1)^{\alpha_d}\epsilon_2\right)^\top\in [-\epsilon_1,\epsilon_1]^m\times [-\epsilon_2,\epsilon_2]^n$, proves \eqref{eqexpqXt}.
\end{proof}

Note that if $m(d\xi)$ and $\mu_i(d\xi)$ have light tails of the
order $e^{-r\|\xi\|^2}d\xi$ for some $r>0$, or have compact
support in particular, then the first assumption of
Theorem~\ref{thmmoments} is satisfied for $V=\mathbb R^d$. Even
then, however, the solution $\psi(t,u)$ exists only on a finite time
horizon $t<T<\infty$ for any nonzero $u\in\mathbb R^d$ in general.
We refer to the discussion of the diffusion case in \citet{mayerhoferfilipovic09},
see also \citet[Chapter 10]{filipovicbook2009}.

We further present an additional result which concerns the existence
of the marginal transition density of integrated affine jump-diffusions,
which are not covered by Theorem \ref{thmtailnew}. In fact, if $X$ is a one-dimensional
affine process on $\mathbb R_+$, then the two-dimensional process $(dX, X\, dt)^\top$ is
affine again, with state space $\mathbb R_+^2$.
However, its diffusion matrix is degenerate and thus violates the conditions of Theorem~\ref{thmtailnew}. Nevertheless, a slight adaption of its proof yields the existence of the marginal transition density of the
integrated process $\int X\,dt$ under some more stringent conditions. The proof of the following theorem is given in Appendix \ref{proof: int cbi dens}.
\begin{thm}\label{th: int cbi dens}
Let $X$ be an $\mathbb R_+$-valued affine process with parameters $(a=0,\alpha,b,\beta,m,\mu)$.
Further, let $p$ be a nonnegative integer with
\begin{equation}\label{eq:jakobsebastian1}
p< \frac{b}{2\alpha}-1.
\end{equation}
Then $\int_0^t X_s ds\, | \, X_0=x$ admits a a density $g(\xi)$ of
class $C^p$ with support in $\mathbb R_+$ and the partial derivatives of $g(\xi)$
of orders $0,\dots,p$ tend to $0$ as $  \xi \to\infty$.
\end{thm}

From the application point of view, we can rephrase the statements of the preceding theorems as follows:
\begin{cor}\label{coraffineC}
Theorems~\ref{thmtailnew}--\ref{th: int cbi dens} provide conditions in terms of the parameters of the affine process $X$ such that the assumptions in Lemmas~\ref{lemass2simple}--\ref{lemass2}, and thus eventually the validity of Assumptions~\ref{ass1} and \ref{ass2}, can explicitly be verified for the density of the (marginal) transition distributions of $X$.
\end{cor}


\section{Examples of Auxiliary Density Functions}\label{sec:weightfunctions}

For our applications of the polynomial density approximations to affine models we shall make the following specific choices for the auxiliary density function $w$. For positive
coordinates we use the Gamma density
\begin{equation}\label{eq:gammadensity}
 \gamma(\xi;D)=\frac{e^{-\xi}\xi^{D}}{\Gamma\brak{1+D}}
\end{equation}
of a $\Gamma(1+D, 1)$-distributed random variable. Here, $\Gamma [\cdot]$ denotes the Gamma function. It is easily seen that conditions \eqref{eps0mom}--\eqref{lemass2ass} are satisfied for the appropriate parameters $p$, $\epsilon_0$ and $\epsilon_1$, respectively.

 For real-valued coordinates we employ the bilateral Gamma density from \citet{kuechlertappe08a}.
The corresponding family of distributions nests, for example, the Variance Gamma distribution as a special case.
It has very flexible shapes \citep{kuechlertappe08b}. For the purposes of this paper we make use of a constrained,
standardized version with mean centered at zero, unit variance, zero skewness, and excess kurtosis $C>0$.
We denote this standardized bilateral Gamma distribution with $\Gamma _b(C)$. Its characteristic function is given by
\[ \Phi _{\Gamma _b}(u;C)=216^{\frac{1}{C}} \left(\frac{1}{6-C u^2}\right)^{3/C} , \quad u\in \reals, \quad C\in \mathbb{R}_+.\]
By \citet[][eq. 3.6]{kuechlertappe08b} the bilateral Gamma distribution has a density given by
\begin{equation}\label{eq:bilatgammadensity}
 \gamma_b(\xi; C)=\frac{2^{\frac{3 (C-2)}{4 C}} 3^{\frac{C+6}{4 C}} C^{-\frac{C+6}{4 C}} |\xi |^{\frac{3}{C}-\frac{1}{2}}
   K_{\frac{3}{C}-\frac{1}{2}}\left(\frac{\sqrt{6} |\xi |}{\sqrt{C}}\right)}{\sqrt{\pi }\, \Gamma \left(\frac{3}{C}\right)},
\end{equation}
where $K_n(\xi)$ denotes the modified Bessel function of the second kind. It follows from \citet[Section 6]{kuechlertappe08b} that conditions \eqref{eps0mom}--\eqref{lemass2ass} are satisfied for the appropriate parameters $\epsilon_0$ and $\epsilon_2$, respectively.
The special case with excess kurtosis $C=1/3$ leads to the following simple expression for the density $\gamma_b(\xi)=\gamma_b(\xi\,;C=1/3)$, since for half-integer indices the modified Bessel
functions evaluate to elementary functions
\begin{multline}\label{eq:bilatgammadensityconstrained}
 \gamma_b(\xi)=\frac{27 e^{-3 \sqrt{2} |\xi|}}{1146880 \sqrt{2}} \cdot \left(7776 \sqrt{2} |\xi|^7+83160 \sqrt{2} |\xi|^5+180180 \sqrt{2} |\xi|^3 \right.\\
 \left. +75075 \sqrt{2} |\xi|  +1296 \xi^8+45360 \xi^6+207900 \xi^4+210210 \xi^2+25025\right).
\end{multline}


Orthonormal polynomial bases can be constructed for any auxiliary density function $w$ which has finite exponential moment \eqref{eps0mom} by the following Gram-Schmidt process, which is also used in the proof of Lemma \ref{lempolydense}.

\begin{algo}[Gram-Schmidt Process]\label{algo:gramschmid}
 \begin{align*}
  H_0&=1 \\
  \widetilde H_\alpha&=\xi^\alpha-\sum _{0\leq|\beta|\leq|\alpha|,\,\beta\neq\alpha} \bracket{\xi ^\alpha,   H_\beta(\xi)}_{\mathcal L^2_w}  H_\beta(\xi)\\
  H_\alpha&=\frac{\widetilde H_\alpha}{\norm{\widetilde H_\alpha(\xi)}_{\mathcal L^2_w}}\quad \textrm{  (normalization) .}
 \end{align*}
\end{algo}
Notice that $\deg H_\alpha = \deg \xi^\alpha = |\alpha|$, which is due to the linear independence of the set of monomials $\{\xi^\alpha\mid\alpha\in\mathbb N_0^d\}$. Below are the first five orthonormal polynomials for the Gamma and the bilateral Gamma densities $\gamma$ and $\gamma_b$ introduced above.

\begin{example}\label{exGamma}
The non-normalized orthogonal polynomials for the Gamma density $\gamma$ are the generalized Laguerre polynomials, the first five of which are
\begin{equation}\label{eq:laguerrepoly}
\begin{split}
  H_0^\gamma(\xi)&=1, \\
\widetilde H_1^\gamma(\xi)&= -\xi+D +1\\
\widetilde H_2^\gamma(\xi)&= \frac{1}{2} \left(\xi^2-2 \xi (D +2)+D ^2+3 D +2\right) \\
\widetilde H_3^\gamma(\xi)&= \frac{1}{6} \left(-\xi^3+3 \xi^2 (D +3)-3 \xi \left(D ^2+5 D
   +6\right)+D ^3+6 D ^2+11 D +6\right) \\
\widetilde H_4^\gamma (\xi)&=\frac{1}{24} (\xi^4-4 (D +4) \xi^3+6 (D +3) (D +4) \xi^2\\
          &-4 (D +2) (D +3) (D +4) \xi+(D +1) (D +2) (D
   +3) (D +4))n.
\end{split}
\end{equation}
The normalization constants are given by
\begin{equation}\label{eq:laguerreortho}
 HO_n^\gamma=\norm{\widetilde H^\gamma_n(\xi)}_{\mathcal
  L^2_\gamma}=\sqrt{\frac{\prod_{i=1}^n(i+D)}{n!}}.
\end{equation}
\end{example}

\begin{example}\label{exbGamma}
For the standardized bilateral Gamma density $\gamma_b$ in \eqref{eq:bilatgammadensity}, the first five non-normalized orthogonal polynomials are
\begin{equation}\label{eq:bilatgammapoly}
\begin{split}
 H_0^{\gamma _b}(\xi)&=1 \\
\widetilde H_1^{\gamma _b}(\xi)&= \xi \\
\widetilde H_2^{\gamma _b}(\xi)&= \xi ^2-1 \\
\widetilde H_3^{\gamma _b}(\xi)&= (-C-3) \xi +\xi ^3 \\
\widetilde H_4^{\gamma _b}(\xi)&=-\frac{2 \left(5 C^2+21 C+18\right) \left(\xi ^2-1\right)}{3
   (C+2)}-C+\xi ^4-3,
\end{split}
\end{equation}
and the corresponding normalization constants $HO_n^{\gamma _b}=\norm{\widetilde H^{\gamma_b}_n(\xi)}_{\mathcal
  L^2_{\gamma_b}}$ are given by
\begin{equation}\label{eq:bilatgammaortho}
\begin{split}
 HO_0^{\gamma _b}&=1 \\
HO_1^{\gamma _b}&= 1 \\
HO_2^{\gamma _b}&= \sqrt{C+2} \\
HO_3^{\gamma _b}&= \sqrt{\frac{7 C^2}{3}+9 C+6} \\
HO_4^{\gamma _b}&=\sqrt{\frac{2 \left(55 C^4+363 C^3+822 C^2+756 C+216\right)}{9 (C+2)}}.
\end{split}
\end{equation}
\end{example}

\begin{example}[Product Measure]\label{ex: product measure}
Define the product density $w_{\gamma \gamma _b}$ with support on $\reals _+ \times \reals$  by
\begin{equation*}
 w_{\gamma\gamma_b}(\xi_1, \xi_2; C, D)=\gamma(\xi_1, D)\gamma_b(\xi_2, C)
\end{equation*}
with Gamma density $\gamma$ defined in \eqref{eq:gammadensity} and bilateral Gamma density $\gamma _b$ defined in \eqref{eq:bilatgammadensity}. Combining Lemma~\ref{lemprod} and Examples~\ref{exGamma} and \ref{exbGamma}, we obtain the corresponding orthonormal basis of polynomials
$\left\{ H_{n_1}^{\gamma} \cdot
H_{n_2}^{\gamma _b}\mid (n_1,
n_2)\in\mathbb N_0^2\right\}$.
\end{example}

\section{Relation to Existing Approximations}\label{sec:relationexistingapproximations}
In this section we recall facts about closed-form density
approximations from previous literature and relate them to the
density expansions of the present paper. A short summary of the
capabilities and limitations of the different methods is reported in
Table \ref{tab:comparison}.

The closest methodology to the one introduced in Section
\ref{sec:densityapproximations} \citet{aitsahalia02}. One of the key
steps is to transform the original process in such a way that a
Gaussian-weighted $\LS ^2$ expansion converges. This is motivated by
an analogy to the central limit theorem \citep[see
also][Introduction]{aitsahaliayu05}: the sampling interval (time
between observations) $\Delta$ plays the role of the sample size $n$
in the central theorem; conditional on a correct standardization a
$N(0,1)$ density turns out to be the correct limiting distribution
as $n\rightarrow \infty$ (in the central limit theorem) and as
$\Delta \rightarrow 0$ (for the stochastic process). The correcting
Hermite polynomials (the pseudo likelihood ratio) then account for
the fact that $\Delta$ is not 0. \citet{aitsahalia02} applies two
transformations. The first change of variables yields a unit
diffusion process through the Lamperti transform. For univariate
diffusions it can be shown that such a transformation always exists.
This step introduces nonlinearities in the drift. The resulting
process is then centered, and scaled in time. Consequently, a
Gaussian-weighted $\LS ^2$ expansion converges {\it uniformly} to
the true, unknown transition density.  This strong convergence
result--which of course exceeds the mere $\LS ^2 $ convergence-- is
proved by using a representation of the true, unknown, transition
density in terms of a Brownian bridge functional from
\citet{Rogers}. Due to the nonlinearities in the drift, the
coefficients of the \citet{aitsahalia02} Hermite expansion are
generally not known in closed form, however. In practice they are
approximated using a Taylor expansion in time in terms of the
infinitesimal generator of the process. This is a key difference to
the setting of the present paper, where expansions are constructed
precisely such that their coefficients are linear in polynomial
moments--and those are known for affine processes without
approximation error.

In the multivariate case, however,  a Lamperti transform is rarely possible, since most applications call for stochastic volatility models which are irreducible \citep[][Proposition 1]{aitsahalia08}. An entirely different strategy is therefore pursued in \citet{aitsahalia08} for the irreducible multivariate case, where the log likelihood is expanded in, both, time, and space, so that the coefficients of the expansion may be computed from the Kolmogorov forward and backward equations. This approach is adopted by \citet{yu07} with the difference that he also considers jump-diffusions and approximates the transition density itself, rather than the log transition density.

The saddlepoint approach in \citet{aitsahaliayu05} is fundamentally
different. It (approximately) solves the Fourier inversion problem
by expanding the cumulant generating function about the
saddlepoint\footnote{For a stochastic process $X$ denote by $K(t,
u\mid x_0)=\log \mathbb E\left[ e^{ u^\top X_t} \mid X_0=x_0\right]$
the cumulant generating function.  Suppose  $X_t|X_0=x_0$ has an
absolutely continuous law. For any state $x$ the saddlepoint is
defined as the solution $\hat{u}=\hat{u}(t, x, x_0)$ in $u$ to the
implicit equation $\partial _u K(t, u\mid x_0)=x$.}, rather
than making use of the Kolmogorov forward and backward equations.
The maintained assumption here is that the cumulant generating
function is available, even though for diffusions \citet[Section
4]{aitsahaliayu05} circumvent this problem by using a Taylor series
expansion for small times along the lines of \citet{aitsahalia02}
for nonlinear moments. Though the saddlepoint approach and this
paper both facilitate expansion techniques, the objects of the
expansion are different and the formulae are unrelated. Saddlepoint
approximations are extremely accurate even for low orders
\citep[][Fig. 2]{aitsahaliayu05}. The price to be paid for this
precision is the computational burden of having to solve numerically
for the saddlepoint for every pair of forward and backward
variables.
\begin{table}
 \begin{tabular}{lccccc}
\hline \hline
 & \multicolumn{4}{c}{Approximations} \\
    & AS02 & ASY05 & Y07 & AS08 & this paper \\
\hline
multivariate & No & Yes & Yes & Yes  & Yes \\
everywhere positive & No & No & No & Yes & No \\
integrates to one & No & No & No & No & Yes  \\
jumps & No & Yes & Yes & No & Yes \\
\hline \hline
\end{tabular}
\caption{{\bf Comparison of Closed-Form Transition Density Approximations. }AS02 refers to \citet{aitsahalia02}, ASY05 to \citet{aitsahaliayu05}, Y07 to \citet{yu07}, and AS08 to \citet{aitsahalia08}.}
\label{tab:comparison}
\end{table}

\section{Applications}\label{sec:applications}
In the following we present applications which highlight the
usefulness of the transition density approximations developed in
this paper. For the empirical investigations considered below we
find that there is a trade-off between numerical accuracy and the
order of the expansion. Higher-order expansions may perform worse
than low-order expansions due to numerical errors that are  induced
by the limited numeric precision of the computer environment in
representing very large or very small numbers. As a general
guideline we suggest matching as many moments (cumulants) as
possible when choosing $\LS ^2$ weights, and stopping the expansion
at a relatively low order such as $J=4$. For the present section  we adopt notation  used conventionally in finance and econometrics. In particular we  deviate from \citet{duffiefilipovicschachermayer03} notation. The time interval between  between observations is generally denoted by $\Delta$.

We strongly recommend checking the above theoretical foundations for the validation of Assumptions \ref{ass1} and \ref{ass2} in numerical applications, as outlined in Example
\ref{example:laguerre} below.



\subsection{Basic Affine Jump-Diffusion (BAJD)}\label{sec:bajd}
We first consider a square-root process with exponentially distributed jumps.  This process has
been recently used in papers that study portfolio credit risk
\citep{duffiegarleanu01, mortensen06,eckner09, feldhutter08} where
it is termed basic affine jump-diffusion (BAJD), and in a bivariate
form in single-name credit \citep{schneidersoegnerveza09}. It can be described in SDE form
\begin{equation}\label{eq:ajd}
 dY_t=(\kappa \theta - \kappa  Y_t)\,dt+\sigma \sqrt{Y_t}dW_t +dK_t.
\end{equation}
The intensity of the compound Poisson process $K$ is $l \geq 0$, and
the expected jump size of the exponentially distributed jumps is
$\nu \geq 0$. The set
of parameters we denote by  $\varrho_Y=\brat{\kappa \theta,
\kappa, \sigma, \nu, l}$ and the domain of the process $\domain =\reals _+$. By Theorem \ref{thmtailnew}, $2 \kappa \theta > \sigma ^2$ ensures existence of transition densities.

\begin{example}[Developing an $\LS _\gamma ^{2}$ Expansion for the BAJD]\label{example:laguerre}

To exemplify the necessary steps to develop a density expansion, we consider here an explicit example and compute an order $J=4$ density expansion for the BAJD process from eq. \eqref{eq:ajd}.
The $\mathcal{L}^2$ weight we use here is a Gamma distribution $\Gamma(1+D,1)$ with density function $\gamma$ from
eq. \eqref{eq:gammadensity}.

\paragraph{Step 1. Computing Conditional Moments: } The generator of the BAJD is
\begin{equation*}
\mathcal Af(x)=(\kappa \theta -\kappa x)\frac{\partial f(x)}{\partial x} + \frac{1}{2}\sigma ^2 x \frac{\partial ^{2}f(x)}{\partial x^{2}}+l\int _{\reals _+ }(f(x+\xi)-f(x))\frac{1}{\nu}e^{-\frac{\xi}{\nu}}d\xi.
\end{equation*}
Hence the matrix $Q=(q_{ij})$ from \eqref{eqQdef} relative to the canonical basis $\brat{1,x,x^2,x^3,x^4}$ equals
\[Q=
 \left(
\begin{array}{ccccc}
 0 &  \kappa \theta +l \nu &  2 l \nu ^2  & 6 l \nu ^3 & 24 l \nu ^4 \\
 0  & -\kappa  & \sigma ^2+2 \kappa \theta +2 l \nu &  6 l \nu ^2 & 24 l \nu ^3 \\
 0  & 0 &   -2 \kappa  & 3 \sigma ^2+3 \kappa \theta +3 l \nu & 12 l \nu ^2 \\
  0 & 0 & 0  & -3 \kappa  & 6 \sigma ^2+4 \kappa \theta +4 l \nu \\
 0 & 0  & 0 & 0  & -4 \kappa
\end{array}
\right).
\]
Note the upper-triangular form. A symbolic mathematics software package such as Mathematica or Maple will be able to compute the matrix exponential $e^{Q  \Delta}$ in closed form. The conditional moments $\moment_n(y_0, \Delta, \varrho _Y)=\ev{Y_\Delta^{n} \mid Y_0=y_0, \varrho _Y}$ may then be obtained by plugging into formula \eqref{eq:momentformula}. We obtain the conditional moments as polynomials of order $\leq 4$ in the backward variable $y_0$. Below we will suppress dependence of the moments on $y_0, \Delta, \varrho _Y$ to lighten notation.

\paragraph{Step 2. Scaling the Process and Computing the Coefficients: } Having computed the first four conditional moments, we now introduce a scaled process $\bar{Y}_{\Delta}=\frac{Y_\Delta \moment_1}{\moment_2 - \moment_1^2}$ and set $D=\moment_1^2/(\moment_2-\moment_1^2)-1$. Note that
\[
\mathbb E[\bar{Y}_{\Delta}\mid Y_0=y_0,\varrho_Y]=\mathbb V [\bar{Y}_{\Delta}\mid Y_0=y_0,\varrho_Y]=D+1.
\]
Hence, in view of Lemma \ref{thm:momentmatching} we match the first two moments of $\bar{Y}_{\Delta}$
with the ones of the standardized gamma density $w=\gamma(\xi;D)$ from eq.~\eqref{eq:gammadensity}, since expectation
and variance of $\gamma(\xi;D)$ equal $D+1$ as well.

By using the corresponding orthogonal polynomials $\widetilde H_n^{\gamma}$ from \eqref{eq:laguerrepoly} along with their normalization constant  $HO_n^{\gamma}$ from eq. \eqref{eq:laguerreortho} we obtain
the coefficients of the density approximation \eqref{eq:truncatedexpansion}: for each $n\geq 0$ we have
\begin{equation}
c_n(y_0, \Delta, \varrho _Y)=\frac{\ev{\widetilde H_n^{\gamma}(\bar{Y}_\Delta)\mid Y_0=y_0,\varrho_Y}}{HO_n^\gamma}.
\end{equation}
In particular, the first five coefficients are of the following explicit form:
\begin{align*}
 c_0(y_0, \Delta, \varrho _Y)&=1 \\
c_1(y_0, \Delta, \varrho _Y)&= 0 \\
c_2(y_0, \Delta, \varrho _Y) &= 0 \\
c_3(y_0, \Delta, \varrho _Y) &= \frac{(D+1) \left((D+2) (D+3)-\frac{(D+1)^2 \mu _3}{\mu _1^3}\right)}{\sqrt{6} \sqrt{(D+1) (D+2) (D+3)}} \\
c_4(y_0, \Delta, \varrho _Y) &= \frac{(D+1)^4 \mu _4+(D+4) (D+1) \mu _1 \left(3 (D+2) (D+3) \mu _1^3-4 (D+1)^2 \mu _3\right)}{2 \sqrt{6} \sqrt{(D+1) (D+2) (D+3) (D+4)} \mu _1^4}.
\end{align*}
Note that due to the chosen scaling, the deforming polynomial (the pseudo likelihood ratio) does not contribute to the density approximation for the first two orders as predicted by Lemma \ref{thm:momentmatching}.

\paragraph{Step 3. Verification of Assumption~\ref{ass2}: }
We denote by $g$ and $\bar g$ the density of $Y_\Delta$ and $\bar Y_\Delta$, respectively. The existence of $g$ and therefore of $\bar g$ is ensured by requiring
\[
\kappa \theta > \frac{\sigma ^2}{2}>0,
\]
by Theorem \ref{thmtailnew}. By the same result, $g$ and $\bar g$ are of class $C^p$ for the greatest nonnegative integer $p$ satisfying
\begin{equation}\label{eq:thebound}
 p<\frac{2\kappa\theta}{\sigma^2}-1.
\end{equation}
On the other hand, using Theorem~\ref{thmmoments} one can verify numerically, by solving the corresponding Riccati differential equations, that
\begin{equation}\label{eq latter}
\mathbb E[e^{\bar Y_\Delta}]=\mathbb E[e^{ \frac{(D+1)}{\mu_1}Y_\Delta}]<\infty.
\end{equation}
This implies finite polynomial moments of $\bar g$ and $g$, and therefore justifies the calculations in Steps 1 and 2, in particular. Note that the Gamma density $w(\xi)=\gamma(\xi;D)$ satisfies $\sup_{x\in[0,1]}x^{D}/w(x)<\infty$ and $\sup_{x\ge 1} \e^{-x}/w(x)<\infty$. In view of \eqref{eq latter}, Lemma~\ref{lemass2spec} implies validity of Assumption~\ref{ass2}, that is $\bar g/w \in\mathcal L^2_{w}$, once $D\le 2p$. By \eqref{eq:thebound}, the latter holds if and only if\footnote{$\lceil x\rceil$ denotes the smallest integer which is greater than or equal to $x$.}
\[ \lceil D/2\rceil < \frac{2\kappa\theta}{\sigma^2}-1,\]
which again can easily be checked numerically.

\paragraph{Step 4. Putting Everything Together: }
Accounting for the change of variable $\bar{y}(y)=\frac{y \moment_1}{\moment_2 - \moment_1^2}$, the density proxy equals
\begin{equation}
 g_Y^{(4)}(y\mid y_0, \Delta, \varrho _Y)=\gamma (\bar{y}(y)) \cdot \sum_{i=0}^{4}c_i(y_0, \Delta, \varrho _Y) H_i^\gamma(\bar{y}(y))\cdot \frac{\moment_1}{\moment_2-\moment_1^2}.
\end{equation}
Figure \ref{fig:densitypic} shows how the polynomials $c_i(y_0, \Delta, \varrho _Y) H_i^\gamma(\bar{y}(y))$ in the pseudo likelihood ratio deform the auxiliary density $w=\gamma$ into the right shape.
\end{example}

\subsection{Heston's Model}\label{sec:hestonmodel}
The \citet{Heston1993} stochastic variance model  has been particularly used for the pricing of equity (index) options. The model for the log stock price $X$ and its stochastic variance $V$ can be realized as solution of the following SDE
\begin{equation}\label{eq:heston}
\begin{split}
dV_t&= (\kappa \theta_V - \kappa_V  V_t)\,dt+\sigma \sqrt{V_t}dW_t^V \\
 dX_t&=(\kappa \theta _X - \frac{1}{2}V_t)dt + \sqrt{V_t}\brac{\rho dW^V_t+\sqrt{1-\rho ^2}dW^X_t},
\end{split}
\end{equation}
with $(W^V, W^X)$ being a two-dimensional standard Brownian motion. The domain $\domain$ of the process equals $\reals _+ \times \reals$. With $2 \kappa \theta _V > \sigma ^{2}$ and $\abs{\rho}<1$ Theorem \ref{thmtailnew}  guarantees existence of transition densities. Note
that it would be perfectly possible to enrich Heston's model above
with jumps in both factors (this has been done for example in
\citet{duffiepansingleton00}, \citet{erakerjohannespolson03}, and
\citet{eraker04}), to multiple variance factors, or even a
matrix-valued variance process as in
\citep{fonsecagrassellitebaldi08}.

The correlation parameter $\rho$
above is a device to model the leverage effect which is partly
responsible for the skew in option prices. Figure \ref{fig:leverage}
displays, using an order 4 expansion, how the skew of the density
may be altered by decreasing the $\rho$ parameter.
\begin{figure}[ht]
\begin{center}
\input{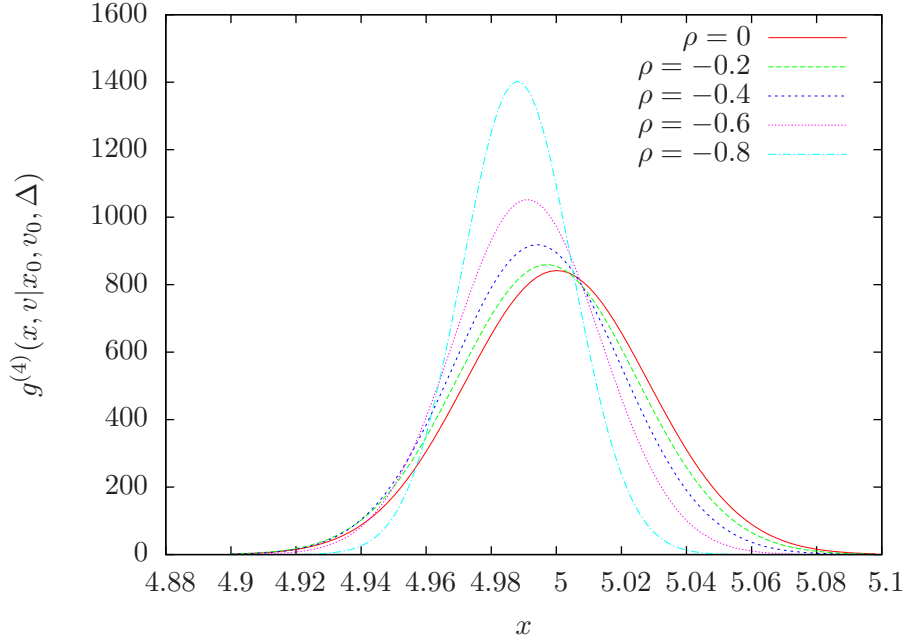}
\caption{\label{fig:leverage}\textbf{The Effect of Leverage: } The figure shows the effect on skewness of negative correlation between the log stock price and its stochastic variance. Depicted are approximated transition densities $g^{(4)}_{VX}$ of the Heston model for different values of the correlation parameter $\rho$ for fixed $v=0.043$. The parameters that generated the picture were  $\Delta=1/52, \kappa \theta_V=0.04, \kappa_V=1, \sigma=0.2, \kappa \theta _X=0.03, x_0=5, v_0=0.04$.}
\end{center}
\end{figure}

For the bivariate Heston model, to compute conditional moments up to order two using formula \eqref{eq:momentformula}, the canonical basis is given by $\brat{1, v, x, v^2, vx, x^2}$ and the corresponding $Q$ matrix, analogous to \eqref{eqQdef}, is
\[ Q =\left(
\begin{array}{cccccc}
 0 & \kappa \theta _V &  \kappa \theta _X & 0 & 0 & 0 \\
 0  & -\kappa _V  & -\frac{1}{2} & \sigma ^2+2 \kappa \theta _V & \kappa \theta _X +\rho  \sigma & 1 \\
 0 & 0 & 0 & 0 & \kappa \theta _V & 2 \kappa \theta _X \\
 0 & 0  & 0 & -2 \kappa _V  & -\frac{1}{2} & 0 \\
 0 &  0 & 0  & 0 & -\kappa _V  & -1 \\
 0 & 0 & 0  & 0 & 0 & 0
\end{array}
\right).
\]

\begin{example}[Standardizing and Scaling the Heston Model]\label{example:heston}

Our goal is to work within a $\mathcal{L}^{2}$ space weighted with a  product measure $w:\reals _+ \times \reals \mapsto \reals _{+}$
\begin{equation*}
w(\xi,\eta)d(\xi, \eta)= \gamma (\xi)d\xi  \cdot  \gamma _b(\eta)d\eta,
\end{equation*}
composed of $\Gamma(D+1,1)$ and $\Gamma_b(C)$ densities, from definitions \eqref{eq:gammadensity} and \eqref{eq:bilatgammadensity}, respectively. In this  space, we may use the polynomials from eqs. \eqref{eq:laguerrepoly} and \eqref{eq:bilatgammapoly}. Lemma \ref{lemprod} ensures that the product of the polynomials forms an ONB in the $\LS _w^2$ space. Acknowledging Lemma \ref{thm:momentmatching} we want to make sure that we match as many moments as possible to optimize the quality of the approximation. Below we show how this transformation $\varsigma$ is composed.

We introduce lighter notation by defining $U _t=( V_t,X_t)$ and for the first two moments of $( V_t, X_t)$
\begin{equation*}
\ev{U _t\mid U _0=u , \varrho _{VX}}=\begin{pmatrix}\mu_1\\ \mu_2 \end{pmatrix}
\end{equation*}
and
\begin{equation*}
\vr{U _t \mid U _0=u , \varrho _{VX}}=\begin{pmatrix}a_1&b\\b&a_2\end{pmatrix}.
\end{equation*}
For demeaning and block-diagonalizing define
\begin{equation*}
\varsigma_1(u)= \Upsilon_1 u+\upsilon_1
\end{equation*}
where
\begin{equation*}
\Upsilon_1=\left(\begin{array}{ll}1&0\\-b/a_1&1\end{array}\right), \quad\upsilon_1=\begin{pmatrix}
                                                                                    0 \\ b\mu_1/a_1-\mu_2
                                                                                   \end{pmatrix}.
\end{equation*}
Then $\varsigma_1(U_t)$ has first two moments of the form
\begin{align*}
\ev{\varsigma_1(U _t) \mid U _0=u , \varrho _{VX}}&=\left(\begin{array}{ll}\mu_1\\0\end{array}\right), \\
\vr{\varsigma_1(U _t) \mid U _0=u , \varrho _{VX}}&=\left(\begin{array}{ll}a_1&0\\0&-b^2/a_1+a_2\end{array}\right).
\end{align*}
The next transformation scales the process into the optimal form (according to Lemma \ref{thm:momentmatching})
\begin{equation*}
\varsigma_2(u)= \Upsilon_2 \cdot u,
\end{equation*}
where
\begin{equation*}
\Upsilon_2=\quad\left(\begin{array}{ll}\mu_1/a_1&0\\0&1/\sqrt{-b^2/a_1+a_2}\end{array}\right).
\end{equation*}
Then $\varsigma_2\circ\varsigma_1(U _t)$ has first two moments of the
form
\begin{align*}
\ev{\varsigma_2\circ\varsigma_1(U _t) \mid U _0=u , \varrho _{VX}}&=\left(\begin{array}{ll}\mu_1^2/a_1\\0\end{array}\right), \\
\vr{\varsigma_2\circ\varsigma _1(U _t) \mid U _0=u , \varrho _{VX}}&=\left(\begin{array}{ll}\mu_1^2/a_1&0\\0&1\end{array}\right),
\end{align*}
and choosing $D=\mu_1^2/a_1-1$ the bivariate orthogonal
expansion of the density of $\varsigma_2\circ\varsigma_2(U _t)$ may be performed in
terms of the polynomials introduced in \eqref{eq:laguerrepoly} and \eqref{eq:bilatgammapoly}. By the transformation the polynomial moments up to second order induced by $w$ agree with the moments of $\varsigma_2\circ\varsigma_1(U _t)$ and the moment-matching Lemma \ref{thm:momentmatching} applies up to order 2.
We have used
\begin{equation*}
\varsigma: u\mapsto\Upsilon_2\circ(\Upsilon_1 u-\upsilon_1),
\end{equation*}
and its inverse is
\begin{equation*}
\varsigma^{-1}: u\mapsto\Upsilon_1^{-1}\Upsilon_2^{-1}u+\Upsilon_1^{-1}\upsilon_1.
\end{equation*}
The parameter $C$ in \eqref{eq:bilatgammadensity} is set to the exact excess kurtosis of the transformed log stock process and the expansion may be performed analogously to Example \ref{example:laguerre}.
\end{example}


\subsection{CDO Pricing}\label{sec:cdopricing}
In the reduced-form credit risk framework \citep{lando1998}, we model the stochastic default intensity $\lambda$ of a corporation with a positive process such as \eqref{eq:ajd}. Under the pricing measure $\Qmeas$ the default time $\tau$ of a corporation is then taken to be the first jump of an inhomogeneous Poisson process with intensity $\lambda$. More formally we write the survival probability of a corporation (using the short-hand notation $\evt{t}{\cdot}=\mathbb{E}\brak{\cdot \mid \mathcal{F}_t}$)
\begin{equation*}
 \Qmeas \brak{\tau >T\mid \mathcal{F}_t}=1\!\!1_{\brat{\tau>t}}\evt{t}{e^{-\int_t^T\lambda_udu}}.
\end{equation*}
All expectations are with respect to the risk-neutral pricing measure $\Qmeas$.  For the pricing of portfolio credit derivatives, to introduce dependence between different obligors, \citet{duffiegarleanu01} (and subsequently \citet{mortensen06}, \citet{eckner09}, and \citet{feldhutter08}) introduce a factor intensity model
\begin{equation}\label{eq:factorintensity}
 \lambda _{it}=X_{it}+a_i Y_t,
\end{equation}
where $X_{it}$ is a firm-specific (idiosyncratic) intensity factor, and $Y_t$ is a (systemic) factor common to all obligors $i=1, \ldots, n$. We model both $X$ and $Y$ with independent jump-diffusion processes from eq. \eqref{eq:ajd}. For $n$ obligors we must impose $\sum_{i=1}^{n}a_i=1$ to ensure identifiability (see \citet{eckner09}).

The survival probability of obligor $i$ according to model \eqref{eq:factorintensity} is then due to independence of the factors
\begin{equation}\label{eq:survivalprob}
 \Qmeas \brak{\tau_i >T\mid \mathcal{F}_t}=1\!\!1_{\brat{\tau_i>t}}\evt{t}{e^{-\int_t^TX_{iu} \, du}}\evt{t}{e^{-a_i \int_t^TY_u \, du}}.
\end{equation}

Defining $Z_{t,T}=\int_t^{T}Y_s\, ds$ we may write the default probability of the $i$th obligor \emph{conditional} on $Z_{t,T}$ as
\begin{equation}\label{eq:conditionaldefaultprob}
 q_i(Z_{t,T})=\Qmeas _t\brak{t<\tau_i \leq T\mid Z_{t,T}}=1\!\!1_{\brat{\tau _i>t}}\brac{1-\evt{t}{e^{-\int_t^TX_{iu} \, du}}e^{-a_i Z_{t,T}}},
\end{equation}
and denoting with $P^{n}_{t,T}(k\mid Z_{t,T})$ the conditional probability that $k$ of the first $n$ credits in the portfolio default between $t$ and $T$ the recursive algorithm of \citet{andersensideniusbasu03} then develops the number $k$ of defaults conditional on $Z_{t,T}$ as
\begin{equation}\label{eq:recursion}
\begin{split}
 P^{(0)}_{t,T}(k\mid Z_{t,T})&=1\!\!1_{\brat{k=0}} \\
 P^{(m+1)}_{t,T}(k\mid Z_{t,T})&=q_{m+1}(Z_{t,T})P^{(m)}_{t,T}(k-1\mid Z_{t,T})+(1-q_{m+1}(Z_{t,T}))P^{(m)}_{t,T}(k\mid Z_{t,T}),
\end{split}
\end{equation}
for $0\leq k \leq n$ and $0\leq m < n$. The expressions $\evt{t}{e^{-\int _t^TX_{iu} \, du}}$, $i=1,\ldots, n$ are unproblematic, but computing the \emph{unconditional} default probability
\begin{equation*}
 P^{(n)}_{t,T}(k)=\int P^{(n)}_{t,T}(k\mid Z_{t,T})d\Qmeas(Z_{t, T})
\end{equation*}
involves an integration against the density of $Z_{t, T}$. We can get hold of the distribution of $Z_{t,T}$ by investigating the joint evolution of $Y$ from eq. \eqref{eq:ajd} and the integral over $Y$. We therefore embed $Y$ into the two-dimensional affine process $(Y,Z)$ described by
\begin{equation}\label{eq:embeddedsystem}
\begin{split}
 dY&=(\kappa \theta - \kappa  Y_t)\,dt+\sigma \sqrt{Y_t}dW_t +dK_t\\
 dZ_t&=Y_t dt.
\end{split}
\end{equation}
Note that even though the instantaneous covariance matrix of the process \eqref{eq:embeddedsystem} above is only of rank one, this process is a well-defined affine process in the sense of
\citet{duffiefilipovicschachermayer03} as pointed out also in Section \ref{sec:truedensities}. Existence of the marginal transition density of $Z_{t,T} \mid Y_t$ is shown in Theorem \ref{th: int cbi dens} for $\kappa \theta > \sigma ^2$.


In principle the conditional default probabilities from eq. \eqref{eq:recursion}
may be computed using the moment generating function of $Z_{t,T}$. In real world applications $n$ is typically larger than 100, however, and the expressions become intractably large, even for small $k$. In practice, recursion \eqref{eq:recursion} is therefore computed through numerical integration. A test of our density expansion in this setting may therefore be reduced to the question of  how well we can approximate the true moment generating function. Below we outline how this approximation can be done in closed form.

Denote by $\evtj{t}{f(Z_{t,T})}$ the expectation of $f(Z_{t,T})$ with respect to a $J$-order expansion instead of the true density. Considering the functional form of the expansion \eqref{eq:truncatedexpansion}, to approximate the expressions $\evt{t}{e^{-Z_{t,T}a_i}}, \, i=1,\ldots, n$ we note that we need to perform the computation
\begin{align}
 \evj _t\brak{e^{a Z_{t,T}}}&=\int_{\Rplus} e^{a\xi} w(\xi)\sum_{j=0}^{J}c_j H_j(\xi)\,d\xi \\
&=\sum_{j=0}^{J}c_H(j)\int_{\Rplus} e^{a\xi} \xi^jw(\xi)\,d\xi \label{eq:approximgf},
\end{align}
where $c_H(j)$ is implicitly defined as\footnote{In practice the coefficients may be collected using a symbolic mathematics package such as Mathematica or Maple.}
\begin{equation*}
 \sum_{j=0}^{J}c_j H_j(\xi)=\sum _{j=0}^{J} c_H(j)\xi^j.
\end{equation*}
The chosen $\mathcal L^2$ weight $w$ for the approximating transition density is a Gamma distribution. To compute eq. \eqref{eq:approximgf} for a random variable $Z$ that is Gamma distributed $Z\sim\Gamma(\alpha, \theta)$ we note that
\begin{equation*}
\mathbb{E}\brak{e^{a Z}Z^n}=  \frac{\theta^{-n } \Gamma (\alpha-n ) (1-a \theta)^{n -\alpha}}{\Gamma (\alpha)}, \, n\in \mathbb{N}, a\in \Real,
\end{equation*}
where $\Gamma$ denotes the Gamma function.

Figure \ref{fig:mgfpic} shows that for the order 10 expansion the approximation error is numerically zero. The order 2 expansion also works well, with negligible numeric error.
\begin{figure}
\begin{center}
 \scalebox{0.8}{\input{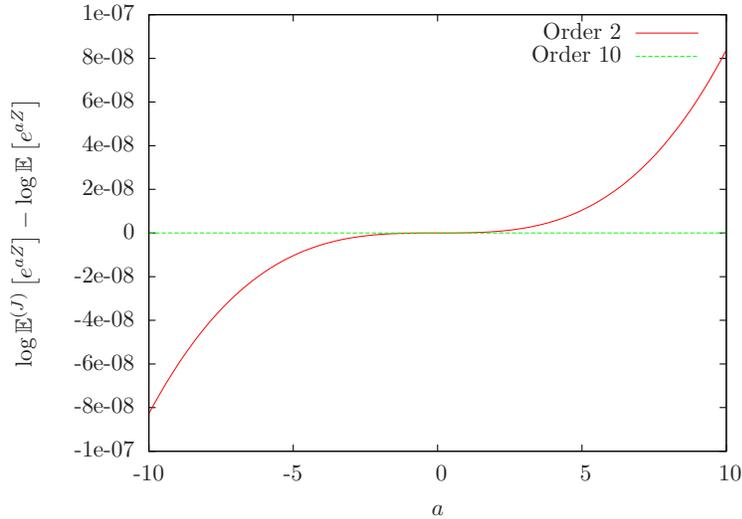}}
\caption{\label{fig:mgfpic}\textbf{True vs. Approximated Moment Generating Function: } The figure shows the log difference between the true moment generating function $\evt{t}{e^{a Z_{t,T}}} $ and the approximated moment generating function $\evtj{t}{e^{aZ_{t,T}}}$  computed for an order 2 and an order 10 expansion of the integrated BAJD from \eqref{eq:embeddedsystem}. The parameters that generated the picture were  $a=1,T-t=5, \kappa \theta=0.00150602, \kappa=0.4648, \sigma=0.01, l=1, \nu=0.0002, y_0=(\kappa \theta + l \nu)/\kappa$. Results are computed using Mathematica and the picture is generated with a numeric precision of 20 digits.}
\end{center}
\end{figure}

\begin{figure}[ht]
\begin{center}
 \subfloat[Order 10 Expansion]{\scalebox{0.8}{\input{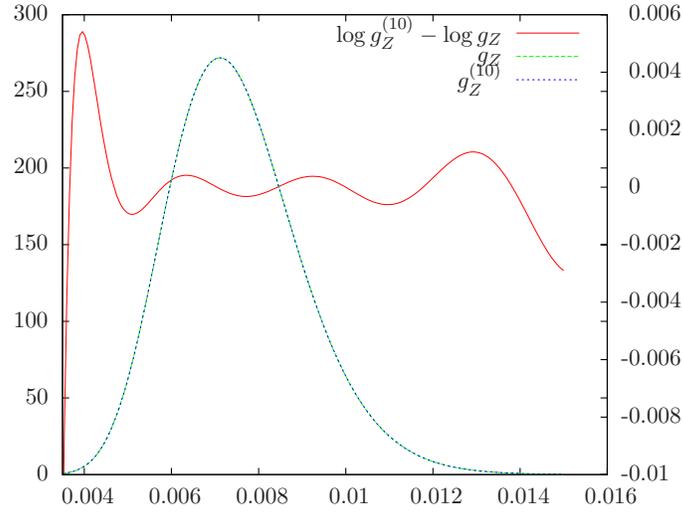}}} \\
 \subfloat[Order 2 Expansion]{\scalebox{0.8}{\input{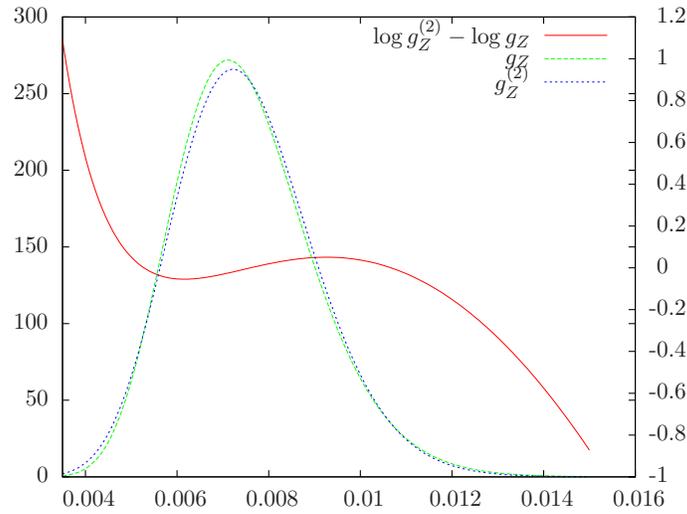}}}
\caption{\label{fig:densitypic}\textbf{Density Plots of the integrated BAJD: } The figure shows the density of $Z_{0,\Delta}\mid y_0, \varrho _Y$ from  specification \eqref{eq:embeddedsystem}. The parameters generating this density are: $\Delta=5, \kappa \theta=0.00150602, \kappa=0.4648, \sigma=0.01, l=1, \nu=0.0002, y_0=(\kappa \theta + l \nu)/\kappa$. The right $y$ axis shows the deviation error to the true density (obtained by Fourier inversion) in percentage terms.}
\end{center}
\end{figure}

\subsection{Likelihood-based Inference}\label{sec:likelihoodinference}
In this section we investigate the performance of the polynomial
density expansions in likelihood-based inference. For discrete,
equally spaced (with time interval $\Delta$) observations $(X_0,
X_1, \ldots, X _N)=X$ of a Markov process $(X_t)_{t \geq 0, X _0=x
_0}$ with domain $\domain$ and parameters $\varrho _X$  we may write
the likelihood function $l_{X}:\domain^{N}\times \varrho_X
\rightarrow \reals _+$ as
\begin{equation}\label{eq:likelihood}
l_{{X}}(X \mid \varrho _X, X_0)=\prod _{i=1}^{N}g_{X}(X_i\mid X _{i-1}, \varrho _{X}, \Delta).
\end{equation}
Denote by
\begin{equation}\label{eq:approxlikelihood}
l_{{X}}^{(J)}(X \mid \varrho _X, X_0)=\prod _{i=1}^{N}g^{(J)}_{X}(X_i\mid X _{i-1}, \varrho _{X}, \Delta)
\end{equation}
the approximate likelihood function using a $J$ order expansion developed in this paper.
The maximum likelihood estimator $\widehat{\varrho} _X$ is obtained as the global maximizer of the likelihood \eqref{eq:likelihood}
\begin{equation}\label{eq:mlestimator}
 \widehat{\varrho} _X=\arg\max _{\varrho _X}\prod _{i=1}^{N}g_{X}(X_i\mid X _{i-1}, \varrho _{X}, \Delta).
\end{equation}
Similarly, for a maximizer obtained from the approximate likelihood $l_{X}^{(J)}$ we write
\begin{equation}\label{eq:qmlestimator}
 \widehat{\varrho} _X^{(J)}=\arg\max _{\varrho _X}\prod _{i=1}^{N}g_{X}^{(J)}(X_i\mid X _{i-1}, \varrho _{X}, \Delta).
\end{equation}
The Bayesian framework \citep[cf.][for reference and comparison to other methodologies]{Robert1994} views the parameters themselves as random variables and is aimed at the posterior density
\begin{equation}\label{eq:posteriordensity}
 \begin{split}
  p(\varrho _X \mid X)&=\frac{l_X(X \mid \varrho _X)}{\int l(X\mid \varrho)\pi(\varrho)d\varrho}\pi(\varrho _X) \\
&\propto l_X(X \mid \varrho _X)\, \pi(\varrho _X) .
 \end{split}
\end{equation}
The prior density $\pi:\varrho _X \rightarrow \reals _+$ expresses the econometrician's personal beliefs and knowledge. Its specification may be fueled by economic intuition, for example that nominal interest rates should be positive, and also  parameter constraints. Note that the expression \eqref{eq:posteriordensity} invokes Bayes theorem and therefore demands that $l_X$ and $\pi$ actually are densities in that they are non-negative functions on the domain of the random variable and integrate to one. This requirement has been challenged to a great extent. The most common violation stems from  expressing uninformedness by setting the prior for a parameter $\varrho _+$ with positive domain proportional to a constant
\begin{equation}\label{eq:approxposterior}
 \pi(\varrho _+)\propto \frac{1}{\sigma _{\varrho _+}}.
\end{equation}
Prior specifications such as the one mentioned above are called  \emph{improper} priors, because their integral does not exist. A less common violation arises from the use of closed-form likelihood expansions within Bayesian inference for Markov processes.\footnote{See \citet{dipietro01} for an introduction to the problem and \citet{stramerbognarschneider09} for MCMC algorithms to overcome it in a very general context.} For the univariate likelihood expansions from \citet{aitsahalia02}  the normalization constant may be evaluated through numeric integration, putting a heavy computational burden on the econometrician. For the multivariate expansions for irreducible models from \citet{aitsahalia08} the normalization constant  does not even exist, because the expansions are purely polynomial. In contrast, the expansions developed in the present paper integrate to one by construction. They share with the expansions from \citet{aitsahalia02} the unpleasant feature that they may become negative, however, even though experience shows that this happens very rarely. For the empirical studies in this paper, for instance,  it has not happened even once.

Subsequently we will denote posteriors where the likelihood is approximated using the approximate likelihood $l_X ^{(J)}$ from \eqref{eq:approxlikelihood} by

\begin{equation*}
p^{(J)}(\varrho _X \mid X)= l_X^{(J)}(X \mid \varrho _X)\, \pi(\varrho _X).
\end{equation*}

To test both methodologies we generate realizations from models \eqref{eq:ajd} and \eqref{eq:heston} through exact simulation methods. We then perform both frequentist and Bayesian inference using our density approximations and the true density (obtained through Fourier inversion of the characteristic function). Frequentist inference is performed on 1,000 data sets generated by model \eqref{eq:heston}, to acquire information about the sampling distribution of the (approximate) maximum likelihood estimators. Bayesian inference is performed on one data set, for the BAJD (eq. \eqref{eq:ajd}) and Heston's model (eq. \eqref{eq:heston}), respectively. We then compare the posterior distribution originating from true density to the posterior distribution generated by the density approximations from this paper.

The simulation for each data set is started from the unconditional mean and then propagated forward 600 data points. We discard the first 100 observations to eliminate impact of the initial condition. To investigate the behavior of our density expansions for different time horizons we choose a monthly observation frequency for the square-root jump-diffusion \eqref{eq:ajd} and weekly observation frequency for the Heston model \eqref{eq:heston}.

To obtain exact draws from the BAJD  we generate exact draws from  $Y_{i} \mid Y_{i-1}$ using \citet[Lemma 2.4]{robertcasella04}. For a uniform random variable $U \sim \mathcal{U}\brac{0,1}$ we exploit that $G_Y^{-1}(U \mid Y_{i-1}, \varrho _Y)\sim G_Y$ for any distribution function $G_Y$. We simulate from \eqref{eq:ajd} using the parameters $\kappa\theta = 0.04, \kappa = 1, \sigma = 0.2, l=3, \nu = 0.01$.

\begin{algo}[Exact draws from BAJD process \eqref{eq:ajd}]\label{algo:exactbajd}
We perform the following procedure starting from $Y_0=\ev{Y_t}$, the unconditional mean, for a realization $Y_i\mid Y_{i-1}$
 \begin{enumerate}
  \item Draw $U\sim \mathcal{U}\brac{0,1}$. Call the realization $u_i$
  \item Use the Newton-Raphson algorithm to compute $y:G_Y (Y\leq y\mid Y_{i-1}, \varrho _Y)=u_i$. In this step we substitute $y=\frac{e^{w}}{1+e^{w}}+c$ to keep $y$ on the positive domain. The floor parameter $c$ we set to $10^{-6}$ to avoid numerical difficulties. The iteration is then
    \begin{equation*}
     w_{j+1}=w_j - e^{-{w_j}} \left(e^{w_j}+1\right)^2 \frac{G_Y(Y\leq c+\frac{e^{w_j}}{e^{w_j}+1}\mid Y_{i-1}, \varrho _Y)-u_i}{g_Y(c+\frac{e^{w_j}}{e^{w_j}+1}\mid Y_{i-1}, \varrho _Y)}
    \end{equation*}
    starting from $w_0= \log \left(\frac{y_{i-1}-c}{c-y_{i-1}+1}\right)$. Stop the iteration at
\begin{equation*}
w^{\star}:\abs{G_Y(Y\leq c+\frac{e^{w^{\star}}}{e^{w^{\star}}+1}\mid Y_{i-1}, \varrho _Y)-u_i}< \varepsilon.
\end{equation*}
Both, $g_Y$, and, $G_Y$ are obtained through Fourier inversion. In our implementation the algorithm terminates after 5 to 6 iterations for $\varepsilon=10^{-6}$.
\item Set $Y_i=c+\frac{e^{w^{\star}}}{e^{w^{\star}}+1}$ increment i and go back to step (1)
 \end{enumerate}
\end{algo}

For Bayesian inference we specify an uninformative prior
\begin{equation}\label{eq:ajdprior}
 \pi(\varrho _{Y})=1\!\!1_{\brat{2 \kappa \theta  > \sigma^2, \sigma>0 , l>0,  \nu >0}}\frac{1}{\sigma  \cdot \kappa \theta \cdot l  \cdot \nu}.
\end{equation}

The Heston parameters are $\kappa_V = 1, \kappa \theta _V = 0.04, \sigma = 0.2, \kappa \theta _X=0.03, \rho=-0.8$.
To obtain exact draws from this model we refer the reader to the algorithm in \citet{broadiekaya06}. For Bayesian inference we specify the prior distribution as
\begin{equation}\label{eq:hestonprior}
 \pi(\varrho _{VX})=1\!\!1_{\brat{2 \kappa \theta _V > \sigma^2,-1<\rho <1, \sigma >0}}\frac{1}{\sigma \cdot \kappa \theta _V}.
\end{equation}
To evaluate the transition density we employ the formulation from \citet{LamoureuxPaseka2005}, which may be evaluated using a single numerical integral, instead of the two-dimensional Fourier integral. This reduction of dimensionality comes at the price of having to evaluate complex-valued special functions, however.


With 1,000 datasets of weekly realizations from the Heston model, for each dataset we obtain parameters $\varrho _{VX}^{\star}$ by maximizing the log likelihood \eqref{eq:likelihood}, respectively the approximate log likelihood \eqref{eq:approxlikelihood}. We use the optimizer \href{http://www.mathematik.tu-darmstadt.de:8080/ags/ag8/Mitglieder/spellucci_de.html}{\texttt{donlp2}} to achieve this task. To relate the density expansions of this paper to existing approximations we perform the estimation experiment with
\begin{itemize}
  \item the true density (obtained through Fourier inversion) denoted by $MLE$
 \item order 4 likelihood expansions developed in this paper using a product measure with a Gamma weight for the variance process and for the log stock variable a
\begin{itemize}
 \item bilateral Gamma weight. Specifically we employ formulation \eqref{eq:bilatgammadensityconstrained}. Estimates are denoted by $BG(4)$
\item Gaussian weight. Estimates are denoted by $G(4)$
\end{itemize}
 \item order 2 closed-form likelihood expansions from \citet{aitsahalia08} denoted by $CF(2)$
\item Gaussian approximation using true conditional moments up to order 2 denoted by $QML$
\end{itemize}
Table \ref{tab:estimationsuccess} reveals that the true likelihood function exhibits problematic behavior for some parameterizations. Only 688 out of 1,000 estimates turned out to be successful. This is due to numerical integration problems that occur in particular for low values of $\sigma$ that arise in the likelihood search. Density expansions developed in this paper are also not entirely unproblematic. Numerical errors from evaluating the pseudo likelihood ratio accumulate and induce spikes that irritate the optimizer's numerical differentiation routines. The Hermite polynomials used for $G(4)$ appear better behaved than the polynomials associated with the bilateral Gamma density used in  $BG(4)$.

Table \ref{tab:hestonbiasmore} reports bias and RMSE of the estimators. The large bias of 0.2255 for the $\kappa$ parameter is a well-established phenomenon that has also been reported in \citet{aitsahaliakimmel05b}. As an overall impression the results suggest that the density approximations developed in this paper exhibit parameter estimates with properties similar to the true ML estimates, while \citet{aitsahalia08} expansions interestingly exhibit lower bias, with the exception of the $\sigma$ parameter, but higher RMSE. In Table \ref{tab:hestonbias} the first column (Mean) in $\widehat{\varrho}_{VX}^{\star MLE}-\varrho _{VX}^{TRUE}$ indicates mean deviation from the true ML estimator and the second column (SD) captures statistical noise in the estimation.  Estimation bias around the MLE for all estimators appears very small. Except for the $CF(2)$ estimator, the noise induced through the density approximations is smaller than the estimation noise of the true $MLE$. Surprisingly,  the QML estimator, a special case of the approximations developed in this paper since it is an order two expansion around a Gaussian, performs remarkably well. All around the BG(4) expansions appear to be the preferable choice. In particular  $\widehat{\sigma}^{\star BG(4)}-\widehat{\sigma}^{\star MLE}$ and  $\widehat{\rho}^{\star BG(4)}-\widehat{\rho}^{\star MLE}$ point to the right direction, the estimators are closer to the true parameters than $MLE$.

The results of the Bayesian inference study also appear promising. Inspecting Figures \ref{fig:hestparms} we see that an order 2 expansion already delivers reasonable results, while the order 4 expansion seems to be even closer to the posterior density obtained from the true density function. To assess how close the posteriors $p_{VX}^{(2)}$ and $p_{VX}^{(4)}$ densities are to the posterior obtained through the true transition density $p_{VX}$ we compute Kolmogorov-Smirnov tests. The results can be seen in Table \ref{tab:ksstats}. They suggest that while $p_{VX}^{(2)}$ appears to be quite different from $p_{VX}$, $p_{VX}^{(4)}$ is statistically almost indistinguishable from the true posterior $p_{VX}$ for the majority of the parameters.

\subsection{Option Pricing}\label{sec:optionpricing}
Heston's model \eqref{eq:heston} is used for option pricing because it may be consistent with the implied volatility skew that can be inferred from market prices. As such it is much more compatible with real data than say, the Black-Scholes model. In stock (index) option pricing the quantity of interest are marginal transition probabilities of the log stock price $X$.
We therefore engineer an approximation directly  around the marginal density of $X_\Delta \mid X_0, V_0$  by expanding $g_{X}$ in $\LS ^2_{\gamma_b}$. We set the constant $C$ from \eqref{eq:bilatgammadensity} to the excess kurtosis of $X$.

Recall that the price of a European call option with maturity $\Delta$ and strike price $K$ is given by
\begin{equation}\label{eq:approxopt}
\begin{split}
 C(\Delta, K)&=e^{-r\Delta}\ev{\brac{e^{X_\Delta}-K}^{+}|X_0=x,V_0=v, \varrho_{VX}} \\
    &=e^{-r \Delta}\brac{\underbrace{\int _{\log K}^\infty e^{\xi} g_{X}(\xi|x, v, \varrho _{VX}, \Delta) d\xi}_{HA(\Delta,K)} - K \underbrace{\int _{\log K}^\infty g_{X}(\xi|x, v, \varrho_{VX}, \Delta) d\xi}_{HB(\Delta,K)}}.
\end{split}
\end{equation}
In accordance with the previous sections we will denote $C^{(J)}(\Delta, K)$, and similarly $HA^{(J)}(\Delta,K)$ and $HB^{(J)}(\Delta, K)$,  the option price computed with $g_{X}^{(J)}$ instead of $g_{X}$. Denoting $\mathbb{Q}(X\leq x)$ the transition probability (and accordingly $\mathbb{Q}^{(J)}(X\leq x)$ the $J$ order approximation of the transition probability) we have that $HB^{(J)}(\Delta,K)=1-\mathbb{Q}^{(J)}(X\leq \log K)$, and using the standardization from Example \ref{example:heston} and the change of variables formula
\begin{equation}\label{eq:hestonb}
\begin{split}
 HB^{(J)}(\Delta,K)&= 1-\int _{-\infty}^{\log K} g^{(J)}_X(\xi|x,v, \varrho_{VX}) d\xi\\
&=1-\frac{1}{\sqrt{a_2}}\int _{-\infty}^{\log K} \gamma_b\brac{\frac{\xi-\mu _2}{\sqrt{a_2}}}\brac{1+\sum _{i=1}^Jc_i(x,v, \varrho _{VX})H^{\gamma _b}_i\brac{\frac{\xi-\mu _2}{\sqrt{a_2}}}}d\xi \\
&=1-\frac{1}{\sqrt{a_2}}\sum _{i=0}^J\Gamma_b\brac{\frac{\log K-\mu _{2}}{\sqrt{a _{2}}},i}\vartheta_i(x,v, \varrho_{VX}).
\end{split}
\end{equation}
Here, $H^{\gamma _b}_i$ are from eqs. \eqref{eq:bilatgammapoly} and \eqref{eq:bilatgammaortho}  and $\vartheta_i(x,v, \varrho_{VX})$ are implicitly defined as
\begin{equation*}
 \brac{1+\sum _{i=1}^Jc_i(x,v, \varrho _{VX})H^{\gamma _b}_i\brac{\frac{\xi-\mu _2}{\sqrt{a_2}}}}=\sum _{i=0}^{J}\vartheta_i(x,v, \varrho_{VX})\xi ^i.
\end{equation*}
The function $\Gamma _b(K,n)=\int _{-\infty}^{K} \xi^n \gamma _b(\xi)d\xi$ is explicit in terms of the Gamma function and regularized generalized hypergeometric functions. The  constituent $HA^{(J)}$ of the approximate call price from eq. \eqref{eq:approxopt} can be computed by numeric integration. Figure \ref{fig:hestdiff} shows that option pricing performance is very good.

\begin{figure}
\begin{center}
 \subfloat[Option Pricing Error\label{fig:optdiff}]{\scalebox{0.5}{\input{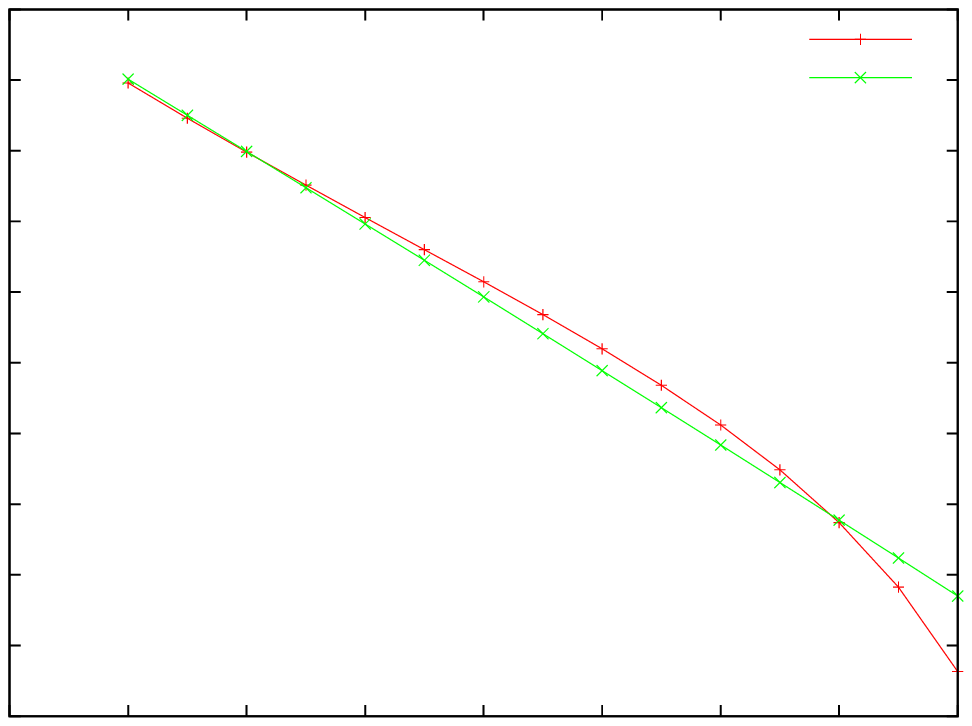}}}
 \subfloat[Density\label{fig:optdens}]{\scalebox{0.5}{\input{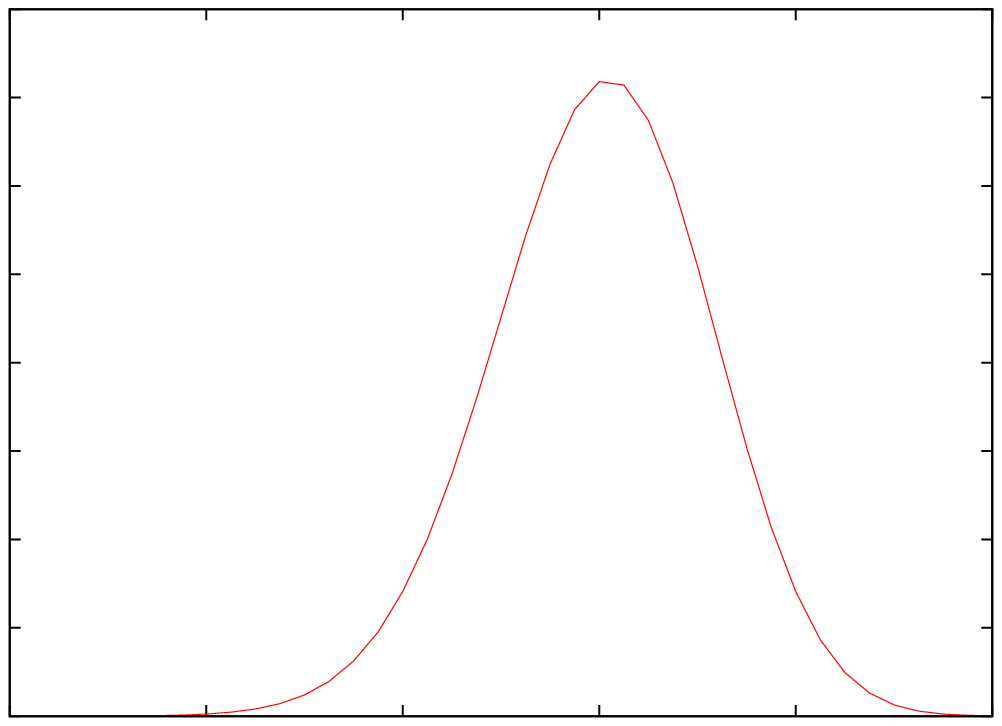}}}
\caption{\label{fig:hestdiff}\textbf{Closed-form option pricing in Heston's model } Panel \ref{fig:optdiff} shows implied Black-Scholes volatility of the true option price, here computed using the \citet{carrmadan99} dampened Fourier inversion approach and the $C^{(4)}(\Delta, K)$ option pricing formula from the approximation of eq. \eqref{eq:approxopt} as a function of strike $K$. The second panel \ref{fig:optdens} shows the density function of $X_\Delta \mid X_0, V_0$ to indicate the likelihood-moneyness trade off. The parameters for the model \eqref{eq:heston} behind the pictures are  $\Delta=1/52, \kappa \theta_V=0.04, \kappa_V=1, \sigma=0.2, \kappa \theta _X=0.03, \rho=-0.8, X_0=5.1, V_0=0.04$.}
\end{center}
\end{figure}
\begin{remark}
 Collecting coefficients  to compute $\vartheta_i$ from Section \ref{sec:optionpricing} eq. \eqref{eq:hestonb}  by hand is very error-prone. Instead we recommend using a symbolic mathematics software package such as Mathematica or Maple.
\end{remark}

\section{Discussion}\label{sec:discussion}

This paper develops a general framework for density approximations for affine processes using orthonormal polynomial expansions in well-chosen weighted $\mathcal L^2$ spaces. We provide novel existence and smoothness results for their transition densities in particular.

The approximations are designed to exploit the explicit polynomial moments of affine processes to compute the coefficients of the expansion without approximation error and in closed form;  the computational burden is concentrated only in the initial calculation of the coefficients of the expansions. Once they are implemented, evaluation is rapid, avoiding the heavy computational cost of Fourier methods to obtain transition densities.  Empirical applications in credit risk, likelihood-based parameter inference, and option pricing suggest that the density expansions are very accurate.

The paper leaves a number of open points for future research. The first question concerns approximations in higher-order weighted Sobolev spaces. One might  suspect that approximation of (sufficiently smooth)
densities in weighted higher-order Sobolev spaces are superior to
$\LS ^{2}$ expansions. In particular, it could be expected that
(i) the quality of approximation might be better (ii) Sobolev
embedding theorems could be applied to infer global uniform
convergence. However, quite
contrary to the $\LS^2$ case, it is unknown
whether the space of polynomials is dense in weighted Sobolev
spaces. Higher-order Sobolev spaces also impose heavy restrictions on the functional form of the approximation weights, which in turn lead to very slow convergence rates. Indeed, preliminary numerical experiments suggest that the price for global, uniform convergence which potentially comes with higher-order Sobolev spaces is a very slow convergence rate.

Another route worth pursuing is a compact truncation of the state space, such that approximations could be performed in non-weighted Sobolev spaces, for which there is more theory available in the literature.

 Suitable approximation weights (such as the bilateral gamma weight of this paper) are a research topic of its own,
and they lead to non-trivial problems in the theory of special functions. Also, density expansions for processes on state spaces different from the canonical ones would be highly desirable. As an example we mention the class of matrix-valued processes
used in co-volatility modeling \citep{leippoldtrojani08,fonsecagrassellitebaldi08,buraschicieslaktrojani08}.

\begin{appendix}

\section{Proofs for Section~\ref{secsuffass}}\label{approofslemma}

This appendix gathers the proofs of the lemmas in Section~\ref{secsuffass}.
\begin{proof}[Proof of Lemma~\ref{lempolydense}]
That the set of polynomials is dense in $\LS^2_w$ is shown in \citet[Lemma 1]{ber_95}. The assumption made in \citet{ber_95} that $w$ is strictly positive can easily be omitted by replacing point-wise equality ``$=0$'' by ``$=0$ $w(\xi)\,d\xi$-a.s.'' at the end of the proof of \citet[Lemma 1]{ber_95}. An orthonormal basis of polynomials $\{H_\alpha\mid\alpha\in\mathbb N_0^d\}$ of $\LS^2_w$ with $\deg H_\alpha = |\alpha|$ is obtained by applying the Gram--Schmidt process to the linearly independent set of monomials $\{\xi^\alpha\mid\alpha\in \mathbb N_0^d\}$, see Algorithm~\ref{algo:gramschmid}.
\end{proof}

\begin{proof}[Proof of Lemma~\ref{lemprod}]
That the product density $w$ on $\mathbb R^d$ has finite exponential moment \eqref{eps0mom} follows from the elementary inequality $\|\xi\|\le \sum_{i=1}^d \left| \xi_i\right|$ for all $\xi\in\mathbb R^d$. The orthonormality of $H_\alpha$ follows from the easily verifiable relationship
\[ \bracket{H_\alpha, H_\beta}_{\LS _w^{2}} = \prod_{i=1}^d \bracket{H^i_{\alpha_i}, H^i_{\beta_i}}_{\LS _{w_i}^{2}(\mathbb R)},\quad \alpha,\,\beta\in \mathbb N_0^d.\]
Moreover, every monomial $\xi^\alpha=\xi_1^{\alpha_1}\cdots\xi_d^{\alpha_d}$ can be written as a product of linear combinations of the respective orthonormal polynomials
\[ \xi_i^{\alpha_i} = \sum_{j=0}^{\alpha_i} \gamma_{ij} H^i_{j}(\xi_i). \]
It follows that the set $\{H_\alpha\mid\alpha\in\mathbb N_0^d\}$ is dense in $\LS^2_w$, and hence forms an orthonormal basis of $\LS^2_w$. This proves the lemma.
\end{proof}

\begin{proof}[Proof of Lemma~\ref{lemass2simple}]
The lemma follows from the estimate
\begin{align*}
 \int_{\mathbb R^d}\frac{g(\xi)^2}{w(\xi)}\,d\xi &= \int_{\mathbb R^d}g(\xi)\,\frac{\e^{-\epsilon_0 \|\xi\|}}{w(\xi)}\,\e^{\epsilon_0 \|\xi\|}\,g(\xi)\,d\xi \\
 &\le  \sup_{x\in\mathbb R^d} \left(g(x)\,\frac{\e^{-\epsilon_0 \|x\|}}{w(x)}\right)\int_{\mathbb R^d}\e^{\epsilon_0 \|\xi\|}\,g(\xi)\,d\xi<\infty.
\end{align*}
\end{proof}

\begin{proof}[Proof of Lemma~\ref{lemass2spec}]
A Taylor expansion of $g(x)$ around $0$ gives $g(x)= \frac{\partial_x^p g(\xi)}{p!} x^p$ for some $\xi=\xi(x)\in [0,x]$. Since $\partial_x^p g$ is continuous we conclude that there exists some finite constant $K$ such that $g(x)\le K\,x^p$ for all $x\in [0,1]$. We then obtain
\begin{align*}
 \int_0^\infty \frac{g(\xi)^2}{w(\xi)}\,d\xi &= \int_0^1 \frac{g(\xi)^2}{w(\xi)}\,d\xi+\int_1^\infty g(\xi)\,\frac{\e^{-\epsilon_0  \xi }}{w(\xi)}\,\e^{\epsilon_0  \xi }\,g(\xi)\,d\xi\\
 &\le K^2\int_0^1 \frac{\xi^{2p}}{w(\xi)}\,d\xi + \sup_{x\ge 1} \left(g(x)\,\frac{\e^{-\epsilon_0 x }}{w(x)}\right)\int_1^\infty \e^{\epsilon_0  \xi }\,g(\xi)\,d\xi<\infty.
\end{align*}
This proves the lemma.
\end{proof}

\begin{proof}[Proof of Lemma~\ref{lemass2}]
Let $x\in\mathcal I$, and let $i^\ast$ be such that $x_{i^\ast}=\min_i x_i$. A Taylor expansion of $g(x,y)$ in $x_{i^\ast}$ around $x_{i^\ast}=0$ gives $g(x,y)=\frac{\partial^p_{x_{i^\ast}} g(\xi,y)}{p!} {x_{i^\ast}}^p$ for some $\xi=\xi(x,y)\in \mathcal I$. Since $\partial^p_{x_{i}}  g(x,y)$ is bounded on $\mathcal I\times\mathbb R^n$ we conclude that there exists some finite constant $K$ such that $g(x,y)\le K\,\min_i x_i^p$ for all $(x,y)\in \mathcal I\times \mathbb R^n$. We then decompose
\[ \int_{\mathbb R^m_+} \int_{\mathbb R^n} \frac{g(\xi,\eta)^2}{w(\xi,\eta)}\,d\xi\,d\eta=I_1+I_2\]
with
\begin{align*}
  I_1&=\int_{\mathcal I} \int_{\mathbb R^n} \frac{g(\xi,\eta)^2}{w(\xi,\eta)}\,d\xi\,d\eta \\
  &\le K\int_{\mathcal I} \int_{\mathbb R^n} \frac{ \min_i\xi_i^p \,\e^{- \epsilon_2 \| \eta\|}}{w(\xi,\eta)}\e^{  \epsilon_2 \| \eta\|}\,g(\xi,\eta)\,d\xi\,d\eta\\
  &\le K\sup_{(x,y)\in {\mathcal I}\times\mathbb R^n} \left(\frac{  \min_i x_i^{p}\,\e^{  - \epsilon_2\| y\|}  }{w(x,y)}\right)\int_{\mathcal I} \int_{\mathbb R^n}  \e^{  \epsilon_2 \| \eta\|}\,g(\xi,\eta)\,d\xi\,d\eta<\infty
\end{align*}
and
\begin{align*}
  I_2&=\int_{(1,\infty)^m} \int_{\mathbb R^n} \frac{g(\xi,\eta)^2}{w(\xi,\eta)}\,d\xi\,d\eta \\
  &= \int_{(1,\infty)^m} \int_{\mathbb R^n} g(\xi,\eta)\,\frac{\e^{-\epsilon_1 \|\xi\| - \epsilon_2 \| \eta\|}}{w(\xi,\eta)}\,\e^{\epsilon_1 \|\xi\| + \epsilon_2 \| \eta\|}\,g(\xi,\eta)\,d\xi\,d\eta\\
  &\le \sup_{(x,y)\in {(1,\infty)^m}\times \mathbb R^n} \left(g(x,y)\,\frac{ \e^{-\epsilon_1 \|x\| - \epsilon_2\| y\|} }{w(x,y)}\right)\int_{(1,\infty)^m} \int_{\mathbb R^n}  \e^{\epsilon_1 \|\xi\| + \epsilon_2 \| \eta\|}\,g(\xi,\eta)\,d\xi\,d\eta\\
  &<\infty.
\end{align*}
This proves the lemma.
\end{proof}

\section{Proof of Theorem \ref{thmtailnew}}\label{sec_proofthmtailnew}

First we note that the existence and smoothness properties of a density on
$\mathbb R^d$ for $X_t|X_0=x$ is invariant with respect
to non-singular linear transformations of the state vector $X_t$. In view of \citet[Theorem 10.7]{filipovicbook2009} there exists a non-singular linear transformation of the state vector $X_t$ mapping $\domain=\mathbb R^m_+\times\mathbb R^n$ onto itself, and which renders block diagonal matrices $\alpha_i$ in the form
\begin{equation}\label{eqblockdiag}
  \alpha_i=\left(\begin{array}{cc}
\diag(0,\dots,0,\alpha_{i,ii},0,\dots,0) & 0 \\ 0 & \alpha_{i,JJ}
\end{array}\right),
\end{equation}
so that $x^\top \alpha_i\, x = \alpha_{i,ii}x_i^2 + x_J^\top \alpha_{i,JJ}\,x_J$ for all $x\in\mathbb R^d$. Moreover, this transformation does not affect the upper diagonal element $\alpha_{i,ii}$ (see the proof of \citet[Lemma 10.5]{filipovicbook2009}), which is important in view of the criterion \eqref{eq:jakobsebastian}. Hence without loss of generality we shall from now on assume that the matrices $\alpha_i$ are of the block diagonal form~\eqref{eqblockdiag}.

We now recall a classical result on characteristic functions $\widehat{\nu}$ of
probability measures $\nu$, see \citet[Proposition 28.1]{Sato}:
\begin{lem}\label{lemdensity}
Let $\nu$ be a probability measure on $\mathbb R^d$. Assume its characteristic function $\widehat{\nu}(\iim u)=\int_{\mathbb R^d} \e^{\iim u\top \xi}\,\nu(d\xi)$ satisfies
  \[ \int_{\mathbb R^d} |\widehat{\nu}(\iim u)|\,\|u\|^k\,du<\infty \]
  for some nonnegative integer $k$. Then $\nu$ has a density $h(x)$ of class $C^k$ and the partial derivatives of $h(x)$ of orders $0,\dots,k$ tend to $0$ as $\|x\|\to\infty$.
\end{lem}

It thus remains to prove the appropriate integrability of the affine characteristic function \eqref{eqatfXX}, that is, the appropriate tail behavior in $u\in\mathbb
R^d$ of the functions $\phi(t,\iim u)$ and $\psi(t,\iim u)$. The following lemma is our core result, which together with Lemma~\ref{lemdensity} completes the proof of Theorem \ref{thmtailnew}.

\begin{lem}\label{lemXXX}
The following properties are equivalent:
\begin{enumerate}
\item \label{lemXXX0} The $d\times (n+1)d$-matrix $\mathcal K$ given in \eqref{Kcaldef} has full rank.
  \item \label{lemXXX1} For any $t>0$, the $d\times d$-matrix
\[ A(t)=\int_0^t \diag(0, e^{\mathcal B_{JJ}^\top s}\, a\,
 e^{\mathcal B_{JJ} s})ds+t\sum_{i\in L}\alpha_i\]
is nonsingular.
  \item \label{lemXXX2} For any $t>0$ there exists an $\epsilon>0$ such that the cones
\begin{align*}
   C_0 &=   \left\{ u\in\mathbb R^d\mid u_J^\top \left(\int_0^t e^{\mathcal B_{JJ}^\top s}\,a\,e^{\mathcal B_{JJ} s}\,ds\right)u_J\ge \epsilon\|u\|^2\right\} \\
C_i&=\left\{ u\in\mathbb R^d\mid u^\top\alpha_i\,u\ge \epsilon\|u\|^2\right\},\quad i=1,\dots,m  \end{align*}
cover $\mathbb R^d$. That is, $\bigcup_{i=0}^m C_i=\mathbb R^d$.
\end{enumerate}
Moreover, any of the above properties, {\ref{lemXXX0}}, {\ref{lemXXX1}}, or {\ref{lemXXX2}}, implies that
 \begin{equation}\label{eqdecaychar}
 \int_{\mathbb R^d} \left|e^{\phi(t,\iim u)+\psi(t,\iim u)^\top x}\right| \|u\|^p\,du<\infty
 \end{equation}
for all nonnegative numbers $p< \min_{i\in\{1,\dots,m\} } \frac{b_i}{\alpha_{i,ii}}-1$.
\end{lem}

The remainder of this section is devoted to the proof of Lemma~\ref{lemXXX}. Let $t>0$ and $u\in\mathbb R^d\setminus\{0\}$. We first claim that $u^\top \mathcal K=0$ if and only if $u^\top A(t)=0$, which proves equivalence of \ref{lemXXX0} and \ref{lemXXX1}. To prove the claim, note that since $A(t)$ is positive semidefinite, $u^\top A(t)=0$ is equivalent to $u^\top A(t)u=0$. Since each of the summands in $A(t)$ is positive semidefinite, this again is equivalent to
\[ u^\top \sum_{i=1}^m  \alpha_i=0 \quad \text{and}\quad u_J^\top e^{\mathcal B_{JJ}^\top s}\,a=0\quad \text{for all $s\in[0,t]$.}\]
A power series expansion of $e^{\mathcal B_{JJ}^\top s}$ shows that this is equivalent to
\[ u^\top \sum_{i=1}^m  \alpha_i=0 \quad \text{and}\quad u_J^\top  (\mathcal B_{JJ}^{k})^\top a=0\quad \text{for all $k\in\mathbb N_0$.}\]
The Cayley--Hamilton theorem (see \cite[Theorem 2.4.2]{hor_joh_85}) implies that, for all $k\ge n$, $\mathcal B_{JJ}^{k}$ is a linear combination of ${\rm Id},\,\mathcal B_{JJ}^{},\dots,\,\mathcal B_{JJ}^{n-1}$. Whence the above property is equivalent to $u^\top \mathcal K=0$, which proves the claim.

The equivalence of \ref{lemXXX1} and \ref{lemXXX2} follows from the identity
\[ u^\top A(t)\,u=u_J^\top \left(\int_0^t e^{\mathcal B_{JJ}^\top s}\,a\,e^{\mathcal B_{JJ} s}\,ds\right)u_J+ t \sum_{i\in L}  u^\top\alpha_i\,u\]
and the fact that each of the summands is nonnegative. This establishes the first part of Lemma~\ref{lemXXX}.

As for the second part of Lemma~\ref{lemXXX}, we note that as a consequence of \eqref{eqblockdiag} the real and imaginary parts
\[f(t,\iim u)=\Re \psi(t,\iim u)\quad\text{and}\quad g(t,\iim u)=\Im \psi(t,\iim u)\]
of $\psi$ satisfy the following system of Riccati equations, for $i=1,\dots,m$:
\begin{equation*}
  \begin{aligned}
    \partial_t f_i&=\alpha_{i,ii}  f_i^2 -  g^\top\alpha_i\, g + \mathcal B_i\,f  +  \int_{\domain}\left(e^{f^\top\xi }\cos\left(g^\top\xi \right)-1\right)\mu_i(d\xi)\\
    f_i(0)&=0\\
    f_J&\equiv 0\\
    \partial_t g_i&= 2  f_i \,\alpha_{i,ii}\, g_i +\mathcal B_{i} \,g  + \int_{\domain} e^{f^\top\xi}\sin\left(g^\top\xi\right)\, \mu_i(d\xi)\\
    g_i(0)&=u_i\\
    g_J&= e^{\mathcal B_{JJ} t} u_J
    \end{aligned}
\end{equation*}
In the sequel we will make use, without further notice, of the fact that $f_i$ is $\mathbb
R_-$-valued for all $i=1,\dots,m$, and that $f_j=0$ for all $j\in J$. In particular, it follows from
above that $f_i$ satisfies the following system of differential
inequalities
\[  \partial_t f_i\le \alpha_{i,ii}  \,f_i^2 - g^\top  \alpha_{i}\,g  + \mathcal B_{ii}\,f_i ,\quad i=1,\dots,m.\]

For any $u\neq 0$, we now define the scaled functions
\begin{equation}\label{FGdefeq}
\begin{aligned}
  F(t,u)&=\frac{1}{\|u\|} f\left(\frac{t}{\|u\|},\iim u\right)\\
  G(t,u)&=\frac{1}{\|u\|} g\left(\frac{t}{\|u\|},\iim u\right).
\end{aligned}
\end{equation}
Then $F$ and $G$ satisfy, for $i=1,\dots,m$:
\begin{equation}\label{odefg}
  \begin{aligned}
    \partial_t F_i&\le\alpha_{i,ii} \,F_i^2 - G^\top  \alpha_{i}\,G + \frac{1}{\|u\|}\mathcal B_{ii}\,F_i \\
    F_i(0)&=0\\
    F_J&\equiv 0\\
    \partial_t G_i&= 2 F_i \,\alpha_{i,ii}\,G_i   +\frac{1}{\|u\|}(\mathcal B_i+\mathcal K_i)\,G \\
    G_i(0)&=\frac{u_i}{\|u\|}\\
    G_J&=e^{\mathcal B_{JJ} t} \frac{u_J}{\|u\|}
  \end{aligned}
\end{equation}
where we define the $d\times d$-matrix $\mathcal K=\mathcal K(t,u)$
by its $1\times d$-row vectors
\[ \mathcal K_i=\begin{cases}
\int_{\domain} \left(\int_0^1   \cos\left(s \|u\| G^\top\xi\right) ds\right)e^{\|u\|  F^\top\xi}\,\xi^\top\mu_i(d\xi),& i=1,\dots,m\\
0,& i=m+1,\dots,d,\end{cases}\] and we have used the simple fact
that
\begin{align*}
  e^{\|u\|  F^\top\xi}\sin\left(\|u\|  G^\top\xi\right)
  &=e^{\|u\| F^\top\xi}\int_0^1 \frac{d}{ds}    \sin\left(s \|u\|  G^\top\xi\right) ds\\
  &=\|u\|  G^\top\xi\,e^{\|u\|  F^\top\xi}\int_0^1      \cos\left(s \|u\| G^\top\xi\right) ds.
\end{align*}
It follows by the assumptions on $\mu_i$ that the matrix $\mathcal
K$ is uniformly bounded
\[ \sup_{t,u}\|\mathcal K(t,u)\|=K<\infty\]
where the constant $K$ only depends on the measures $\mu_i$. The
squared norm of $G$ thus satisfies
\begin{align*}
  \partial_t \|G\|^2 &=2 G^\top\partial_t G =4 G^\top{\rm diag}(F)\,G +\frac{2 }{\|u\|}\, G^\top(\mathcal B+\mathcal K)\,G   \\
  &\le \frac{2 }{\|u\|}\left(\|\mathcal B\|+  K\right)\|G\|^2 \\
   \|G(0)\|^2&=1.
\end{align*}
We shall now and in the sequel make use of the following comparison
result, which is a special case of a more general theorem proved by
\citet{volkmann}:
\begin{lem}\label{lemcomp}
  Let $R(t,v)$ be a continuous real map on $\mathbb R_+\times\mathbb R$ and locally Lipschitz continuous in $v$. Let $p(t)$ and $q(t)$ be differentiable functions satisfying
  \begin{align*}
    \frac{d}{dt} p(t) &\le R(t,p(t))\\
    \frac{d}{dt} q(t) &= R(t,q(t))\\
    p(0)&\le q(0).
  \end{align*}
  Then we have $p(t)\le q(t)$ for all $t\ge 0$.
\end{lem}

Applying Lemma~\ref{lemcomp} to the above differential inequality for $\|G\|^2$ we obtain
\begin{equation}\label{Gnorm}
   \|G\|^2\le  e^{\frac{2 }{\|u\|}\left(\|\mathcal B\|+  K\right)t}.
\end{equation}
For any $i\in\{1,\dots,m\}$ we then obtain the differential
inequality
\begin{align*}
  \partial_t\left( G^\top\alpha_i\,G\right)&=2 G^\top\alpha_i\,\partial_t G\\
  &= 4 G^\top\alpha_i\,{\rm diag}(F)\,\alpha_i\,G +\frac{2 }{\|u\|}\, G^\top\alpha_i(\mathcal B+\mathcal K)\,G\\
  &\ge 4 \alpha_{i,ii} \, F_i\,\alpha_{i,ii}\,G_i^2 -  \frac{2 \|\alpha_i\|}{\|u\|}\left(\| \mathcal B\|+  K\right) \|G\|^2 \\
  &\ge 4 \alpha_{i,ii} \, F_i\,G^\top\alpha_i\,G -  \frac{2 \|\alpha_i\|}{\|u\|}\left(\| \mathcal B\|+  K\right) e^{\frac{2 }{\|u\|}\left(\|\mathcal B\|+  K\right)t}\\
  G(0)^\top\alpha_i\,G(0)&=\frac{u^\top\alpha_i u}{\|u\|^2}
\end{align*}
where we have used the fact that $\alpha_{i,ii}\,G_i^2\le
G^\top\alpha_i\,G$ and \eqref{Gnorm} for the last inequality. Lemma~\ref{lemcomp} again yields the lower bound
\begin{equation*}
\begin{aligned}
  G^\top\alpha_i\,G &\ge e^{4\alpha_{i,ii}\int_0^t F_i(s)\,ds} \frac{u^\top\alpha_i u}{\|u\|^2} \\
  &\quad -\frac{2 \|\alpha_i\|}{\|u\|}\left(\| \mathcal B\|+  K\right)\int_0^t \underbrace{e^{4\alpha_{i,ii}\int_s^t F_i(r)\,dr}}_{\le 1}e^{\frac{2 }{\|u\|}\left(\|\mathcal B\|+  K\right)(t-s)}\,ds\\
  &\ge e^{4\alpha_{i,ii}\int_0^t F_i(s)\,ds} \frac{u^\top\alpha_i u}{\|u\|^2}-\|\alpha_i\| \left(e^{\frac{2 }{\|u\|}\left(\|\mathcal B\|+  K\right)t}-1\right)
\end{aligned}
  \end{equation*}
Combining this with the differential inequality \eqref{odefg} for
$F$ we obtain
\begin{equation}\label{odefnew}
  \begin{aligned}
    \partial_t F_i&\le\alpha_{i,ii} \,F_i^2 + \frac{1}{\|u\|}\mathcal B_{ii}\,F_i- e^{4\alpha_{i,ii}\int_0^t F_i(s)\,ds} \frac{u^\top\alpha_i u}{\|u\|^2} \\
  &\quad +\|\alpha_i\| \left(e^{\frac{2 }{\|u\|}\left(\|\mathcal B\|+  K\right)t}-1\right)  \\
    F(0)&=0.
  \end{aligned}
\end{equation}
We arrive at the following intermediate result.
\begin{lem}\label{lemminrho}
For every $\epsilon>0$ and $t_0>0$ there exists some $\rho>0$ and
$R>0$ such that
\begin{equation*}
  F_i(t_0,\iim u)\le -\rho
\end{equation*}
for all $u\in\mathbb R^d$ with $\|u\|\ge R$ and
$u^\top\alpha_i\,u\ge \epsilon\|u\|^2$, for all $i\in\{1,\dots,m\}$.
\end{lem}
\begin{proof}
The differential inequality \eqref{odefnew} is autonomous and smooth
in $F_i$. Moreover, the initial slope satisfies
\[\partial_t F_i(t,\iim u)|_{ t=0 }\le -\frac{u^\top\alpha_i u}{\|u\|^2}\le -\epsilon\]
uniformly in $i$ and $u$ in the designated set. Also notice the
estimate
\[ \frac{1}{\|u\|}\mathcal B_{ii}\,F_i\le -\frac{1}{R}|\mathcal B_{ii}|\,F_i \]
and the uniform bound on the last summand on the right hand side of
\eqref{odefnew} for $t\le t_0$ and $\|u\|\ge R$. The claim now
follows from Lemma~\ref{lemcomp}.
\end{proof}

Below we shall make use of the following is easy to check auxiliary result on Riccati equations:
\begin{lem}\label{ric sol}
Let $A>0, B\in\mathbb R\setminus\{0\}$, and $t_0\geq 0$, and
$G_0<0$. For $t\geq t_0$, the solution of
\[\partial_t G(t)=AG(t)^2+B G(t),\quad G(t_0)=G_0
\]
is of the form
\[
G(t)=\frac{BG_0 e^{B(t-t_0)}}{\left(AG_0+B\right)-A
G_0e^{B(t-t_0)}}.
\]
If $B=0$, then
\[
G(t)=\frac{G_0}{1-AG_0(t-t_0)}.
\]
\end{lem}

From \eqref{odefg} we deduce the trivial differential inequality
\begin{equation*}
  \begin{aligned}
    \partial_t F_i&\le\alpha_{i,ii} \,F_i^2 + \frac{\mathcal B_{ii}}{\|u\|}\,F_i.
  \end{aligned}
\end{equation*}
By Lemma \ref{ric sol}, the solution of
\[ \partial_t h = \alpha_{i,ii} \,h^2 + \frac{\mathcal B_{ii}}{\|u\|}\,h\]
with $h(t_0)<0$ is explicitly given by
\[ h(t) = -\frac{e^{\frac{\mathcal B_{ii}}{\|u\|}(t-t_0)}}{ \|u\|\frac{\alpha_{i,ii}}{\mathcal B_{ii}}\left(e^{\frac{\mathcal B_{ii}}{\|u\|}(t-t_0)}-1\right)-\frac{1}{h(t_0)}},\quad t\ge t_0.\]
Together with Lemmas~\ref{lemcomp} and \ref{lemminrho} and we thus obtain that
\begin{equation}\label{Fituineq}
 F_i(t,\iim u)\le  -\frac{e^{\frac{\mathcal B_{ii}}{\|u\|}(t-t_0)}}{\|u\|\frac{\alpha_{i,ii}}{\mathcal B_{ii}}\left(e^{\frac{\mathcal B_{ii}}{\|u\|}(t-t_0)}-1\right)+\frac{1}{\rho}},\quad t\ge t_0
\end{equation}
for all $\|u\|\ge R$ with $u^\top\alpha_i\,u\ge \epsilon\|u\|^2$. By
rescaling we infer
\begin{align*}
  f_i(t,\iim u)&=\|u\| F_i(t\|u\|,\iim u)\\
  &\le - \frac{\|u\|e^{\mathcal B_{ii}\left(t- \frac{t_0}{\|u\|}\right)}}{\|u\|\frac{\alpha_{i,ii}}{\mathcal B_{ii}}\left(e^{\mathcal B_{ii}\left(t- \frac{t_0}{\|u\|}\right)}-1\right)+\frac{1}{\rho} }\\
  &=-\frac{1}{\alpha_{i,ii}}\frac{d}{dt} \log\left(\|u\|\frac{\alpha_{i,ii}}{\mathcal B_{ii}}\left(e^{\mathcal B_{ii}\left(t- \frac{t_0}{\|u\|}\right)}-1\right)+\frac{1}{\rho} \right) ,\quad t\ge \frac{t_0}{\|u\|}
\end{align*}
for all $\|u\|\ge R$ with $u^\top\alpha_i\,u\ge \epsilon\|u\|^2$.
Integrating this inequality yields
\begin{equation}\label{eqpreeq}
 \begin{aligned}
  \int_0^t f_i(s,\iim u)\,ds &\le \int_{\frac{t_0}{\|u\|}}^t f_i(s,\iim u)\,ds\\
  &= -\frac{1}{\alpha_{i,ii}}\left[ \log\left(\|u\|\frac{\alpha_{i,ii}}{\mathcal B_{ii}}\left(e^{\mathcal B_{ii}\left(t- \frac{t_0}{\|u\|}\right)}-1\right)+\frac{1}{\rho} \right)-\log\left(\frac{1}{\rho}\right)\right]\\
  &=-\frac{1}{\alpha_{i,ii}}\log\left(\rho\|u\|\frac{\alpha_{i,ii}}{\mathcal B_{ii}}\left(e^{\mathcal B_{ii}\left(t- \frac{t_0}{\|u\|}\right)}-1\right)+1 \right),\quad t\ge \frac{t_0}{\|u\|}
\end{aligned}
\end{equation}
for all $\|u\|\ge R$ with $u^\top\alpha_i\,u\ge \epsilon\|u\|^2$. We arrive at the following key result, which completes the proof of Lemma~\ref{lemXXX}.
\begin{lem}\label{lemkey1}
Let $i\in\{1,\dots,m\}$. For every $\epsilon>0$ and $t>0$ there exists some $R>0$ and $C>0$
such that
\[ \left|e^{\phi(t,\iim u)+\psi(t,\iim u)^\top x}\right|\le C \left(1+\|u\|\right)^{-\frac{b_i}{\alpha_{i,ii}}}\]
for all $u\in\mathbb R^d$ with $\|u\|\ge R$ and
$u^\top\alpha_i\,u\ge\epsilon\|u\|^2$. Moreover,
\begin{equation*}
  \Re\phi(t,\iim u)\le -u_J^\top \left(\int_0^t e^{\mathcal B_{JJ}^\top s}\,a\,e^{\mathcal B_{JJ} s}\,ds\right)u_J
\end{equation*}
for all $u\in\mathbb R^d$.
\end{lem}

\begin{proof}
Integration of \eqref{ricceq} implies
\begin{align*}
  \Re\left(\phi(t,\iim u)+\psi(t,\iim u)^\top x\right)&\le \Re\phi(t,\iim u)\\
&\le -u_J^\top \left(\int_0^t e^{\mathcal B_{JJ}^\top s}\,a\,e^{\mathcal B_{JJ} s}\,ds\right)u_J + b_i \int_0^t f_i(s,\iim u)\,ds.
\end{align*}
Together with \eqref{eqpreeq}, this proves the lemma.
\end{proof}

\section{Proof of Theorem \ref{th: int cbi dens}}\label{proof: int cbi dens}
Fix some $\gamma>0$. We consider the two-dimensional process $Z=(X,Y)$,
where $Y_t=y+\gamma\int_0^t X^x_s ds$, with $y\in\mathbb R_+$. It is easy
to see that $(X,Y)$ is an affine process with state space $\mathbb
R^2_+$. In particular, if $y=0$ we have that $Y_t=\gamma \int_0^t
X^x_s$ has an exponentially affine characteristic function of the
form
\[
\mathbb E\left[e^{\iim v Y_t}\mid X_0=x\right]=e^{\phi(t,\iim v)+\psi(t,\iim v)x}, \quad
v\in \mathbb R,
\]
where the characteristic exponents $\phi$ and $\psi$ satisfy the generalized Riccati
differential equations
\begin{align*}
\partial_t\phi&=b\psi\\
\phi(0)&=0\\
\partial_t\psi&=\alpha\psi^2+\beta\psi+\iim\gamma v+\int_0^\infty \left(e^{\psi\xi}-1\right)\mu(d\xi)\\
\psi(0)&=0.
\end{align*}
For any $v\neq 0$, we define the scaled functions\footnote{Note that here we have to scale by $\sqrt{|v|}$, which is in contrast to the proof of Theorem \ref{thmtailnew}, see \eqref{FGdefeq}.}
\begin{align*}
  F(t,\iim v)&=\frac{1}{\sqrt{|v|}} \Re\psi\left(\frac{t}{\sqrt{|v|}},\iim v\right)\\
  G(t,\iim v)&=\frac{1}{\sqrt{|v|}} \Im\psi\left(\frac{t}{\sqrt{|v|}},\iim v\right).
\end{align*}
Then $F$ and $G$ satisfy
\begin{equation}\label{odefgXX}
  \begin{aligned}
    \partial_t F&=\alpha \left(F^2 - G^2\right) + \frac{\beta}{\sqrt{|v|}}F +\frac{1}{|v|}\int_0^\infty \e^{\sqrt{\vert
v\vert}F\,\xi}\left(\cos\left(\sqrt{\vert
v\vert}F\,\xi\right)-1\right)\mu(d\xi) \\
    F(0)&=0\\
    \partial_t G&= 2 \alpha F G    +\frac{\beta}{\sqrt{|v|}}G + \gamma\frac{v}{|v|} +\frac{1}{|v|}\int_0^\infty \e^{\sqrt{\vert
v\vert}F\,\xi} \sin\left(\sqrt{\vert
v\vert}F\,\xi\right)\,\mu(d\xi) \\
    G(0)&=0
  \end{aligned}
\end{equation}
We now prove a first intermediary result, which is analogous to Lemma~\ref{lemminrho}:
\begin{lem}\label{lemminrhoXX}
There exists some $t_0>0$, $\rho>0$ and
$R>0$ such that
\begin{equation*}
  F(t_0,\iim v)\le -\rho
\end{equation*}
for all $v$ with $|v|\ge R$.
\end{lem}
\begin{proof}
For $v\to \pm\infty$, the solutions $F(t,\iim v)$ and $G(t,\iim v)$ of \eqref{odefgXX} converge locally uniformly in $t$ to the solutions $F_\infty(t)$ and $G_\infty(t)$  of the system
  \begin{align*}
    \partial_t F_\infty&=\alpha \left(F_\infty^2 - G_\infty^2\right)  \\
    F_\infty(0)&=0\\
    \partial_t G_\infty&= 2 \alpha F_\infty G_\infty    \pm 1\\
    G_\infty(0)&=0.
  \end{align*}
Since $\partial_t G_\infty(t)|_{t=0}=\pm 1$ it follows that there exists some $t_1>0$ such that $G_\infty(t)\neq 0$ for all $t\in(0,t_1)$. This again implies that $F_\infty(t_0)<0$ for some $t_0\in (0,t_1)$, and the lemma follows.
\end{proof}

From \eqref{odefgXX} we deduce the trivial differential inequality
\[ \partial F\le \alpha F^2+\frac{\beta}{\sqrt{|v|}}F .\]
Arguing as for the derivation of \eqref{Fituineq}, we then obtain together with Lemmas~\ref{lemcomp} and \ref{lemminrhoXX} that
\[ F(t,\iim v)\le  -\frac{e^{\frac{\beta}{\sqrt{|v|}}(t-t_0)}}{\sqrt{|v|}\frac{\alpha }{\beta}\left(e^{\frac{\beta}{\sqrt{|v|}}(t-t_0)}-1\right)+\frac{1}{\rho}},\quad t\ge t_0\]
for all $v\in\mathbb R$ with $|v|\ge R$. By
rescaling and integrating we infer, arguing as for the derivation of \eqref{eqpreeq}, that
 \begin{align*}
  \int_0^t \Re\psi(s,\iim v)\,ds &\le \int_{\frac{t_0}{\sqrt{|v|}}}^t \Re\psi(s,\iim v)\,ds\\
  &= -\frac{1}{\alpha}\left[ \log\left(\sqrt{|v|}\frac{\alpha}{\beta}\left(e^{\beta\left(t- \frac{t_0}{\sqrt{|v|}}\right)}-1\right)+\frac{1}{\rho} \right)-\log\left(\frac{1}{\rho}\right)\right]\\
  &=-\frac{1}{\alpha}\log\left(\rho\sqrt{|v|}\frac{\alpha }{\beta}\left(e^{\beta\left(t- \frac{t_0}{\sqrt{|v|}}\right)}-1\right)+1 \right),\quad t\ge \frac{t_0}{\sqrt{|v|}}
\end{align*}
for all $v\in\mathbb R$ with $|v|\ge R$. Similarly as in
Lemma~\ref{lemkey1} we now infer that
\[ \left|e^{\phi(t,\iim v)+\psi(t,\iim v)  x}\right|\le C \left(1+|v|\right)^{-\frac{b}{2\alpha }}. \]
Combining this with Lemma~\ref{lemdensity} completes the proof of Theorem~\ref{th: int cbi dens}.

\section{Figures and Tables}\label{app:figuresandtables}

\begin{table}[ht]
\begin{small}
\begin{center}
\begin{tabular}{r c c c r c c}
\hline \hline
\multicolumn{3}{c}{Heston Model} & & \multicolumn{3}{c}{BAJD} \\
 KS & Order 2 &  Order 4  & & KS & Order 2 &  Order 4 \\
\hline
$\kappa \theta _V$ & 0.0199& 0.6654 & &$\kappa \theta _V$ & 0.0038& 0.2360\\
$\kappa  _V$ & 0.0000 & 0.0693 & &$\kappa  _V$ & 0.0396 & 0.7658\\
$\sigma  _V$ & 0.0001 & 0.2574 & &$\sigma  _V$ & 0.0000 & 0.6754 \\
$\kappa \theta   _X$ & 0.0018 & 0.0348 & & $l   $ & 0.0000 & 0.0571\\
$\rho $ & 0.0003 & 0.3291 & &$\nu $ & 0.0000 & 0.0049\\
\hline \hline
\end{tabular}
\caption{\label{tab:ksstats}\textbf{Kolmogorov-Smirnov test statistics: } The table displays p-values for a two-sided Kolmogorov-Smirnov test applied to posterior density $p_{VX}$ to $p_{VX}^{(2)}$ and $p_{VX}^{(4)}$ using prior \eqref{eq:hestonprior} for the Heston model in the left panel. The right panel displays p-values for the test applied to $p_{Y}$ to $p_{Y}^{(2)}$ and $p_{Y}^{(4)}$, the BAJD model. The prior distribution for this model is defined in eq. \eqref{eq:ajdprior}. The true posterior is defined in eq. \eqref{eq:posteriordensity} and the approximate posterior densities are defined in eq. \eqref{eq:approxposterior}.}
\end{center}
\end{small}
\end{table}


\begin{table}[ht]
\begin{small}
\begin{center}
\begin{tabular}{c c c c c c}
\hline \hline
 & MLE & BG(4) & QML & G(4) & CF(2) \\
\hline
\#Success & 688 & 841 & 982 & 949 & 981 \\
\hline \hline
\end{tabular}
\caption{\label{tab:estimationsuccess}\textbf{Heston Estimation Success: } The table reports the number of estimation successes on 1,000 datasets generated as exact draws from the Heston model using the technology from \citet{broadiekaya06}. Estimation success is defined by the optimizer meeting the termination criterion, which is a function of the norm of the gradient of the log likelihood function. The density approximations used are BG(4), a fourth order expansion using a Bilateral Gamma weight for the log stock variable and a Gamma weight for the variance variable, G(4), a fourth order expansion using a Gaussian weight for the log stock variable  and a Gamma weight for the variance variable, QML denotes a Gaussian approximation using the true conditional moments up to order 2, and CF(2) denotes the second-order likelihood expansions from \citet{aitsahalia08}. The optimizer used in the likelihood search is \href{http://www.mathematik.tu-darmstadt.de:8080/ags/ag8/Mitglieder/spellucci_de.html}{donlp2}.}
\end{center}
\end{small}
\end{table}

\begin{landscape}
\begin{table}[ht]
\begin{small}
\begin{center}
\subfloat[Bias\label{tab:hestonbiasmore}]{\begin{tabular}{c  d | R[.][.]{1}{4} R[.][.]{1}{4} | R[.][.]{1}{4} R[.][.]{1}{4} | R[.][.]{1}{4} R[.][.]{1}{4} | R[.][.]{1}{4} R[.][.]{1}{4} | R[.][.]{1}{4} R[.][.]{1}{4}}
\hline \hline
  & \multicolumn{1}{l}{$\varrho _{VX}^{TRUE}$} &\multicolumn{2}{l}{$MLE$} &  \multicolumn{2}{l}{$BG(4)$} &  \multicolumn{2}{l}{$QML$} & \multicolumn{2}{l}{$G(4)$} & \multicolumn{2}{l}{$CF(2)$}\tabularnewline
 & & \multicolumn{1}{c}{Bias} & \multicolumn{1}{c}{RMSE} & \multicolumn{1}{c}{Bias} & \multicolumn{1}{c}{RMSE} & \multicolumn{1}{c}{Bias} & \multicolumn{1}{c}{RMSE} & \multicolumn{1}{c}{Bias} & \multicolumn{1}{c}{RMSE} & \multicolumn{1}{c}{Bias} & \multicolumn{1}{c}{RMSE}  \tabularnewline
\hline
$\kappa$ & 1 & 0.2255346834
  & 0.4253164942
 &0.2312654377
 &  0.4168232685
& 0.2324136176
 &  0.4213507663
& 0.232714891
& 0.4186912919
& 0.1904904107
& 0.5552837343
 \tabularnewline

$\kappa \theta _V$ & 0.04 & 0.0093498208
 & 0.0165699736
& 0.0094309984
 & 0.0164289394
& 0.0092745832
 & 0.0166689832
& 0.0095096777
& 0.0164442244
& 0.0063333792
& 0.0178803785
\tabularnewline

$\sigma$ & 0.2 & 0.0009221419
 &0.0049012974
 & 0.0006725087
 & 0.0046705991
& 0.0008380277
& 0.0047349522
& 0.0009062093
& 0.004724057
& 0.0044054689
& 0.0117324295
\tabularnewline

$\kappa \theta _X$ & 0.03 & -0.0151026889
 & 0.0472526603
 & -0.0145396078
 & 0.0468729585
& -0.013796895
 & 0.0482735556
& -0.0149725637
& 0.047200842
& -0.0050352531
& 0.0623627603
\tabularnewline

$\rho$ & -0.8 & -0.0015771471
  & 0.0146879927
 & -0.0000860346020760525
 & 0.0138269018
& -0.0015087734
 & 0.014047429
& -0.0016289221
& 0.0138297495
& 0.0002704118
& 0.0270464961
\tabularnewline
\hline \hline
\end{tabular}}\\
\subfloat[Estimation Noise\label{tab:hestonbias}]{\begin{tabular}{c  d | R[.][.]{1}{4} R[.][.]{1}{4} | R[.][.]{1}{4} R[.][.]{1}{4} | R[.][.]{1}{4} R[.][.]{1}{4} | R[.][.]{1}{4} R[.][.]{1}{4} | R[.][.]{1}{4} R[.][.]{1}{4}}
\hline \hline
  & \multicolumn{1}{l}{$\varrho _{VX}^{TRUE}$} &\multicolumn{2}{l}{$\widehat{\varrho}_{VX}^{\star MLE}-\varrho _{VX}^{TRUE}$} &  \multicolumn{2}{l}{$\widehat{\varrho}_{VX}^{\star BG(4)}-\widehat{\varrho}_{VX}^{\star MLE}$} &  \multicolumn{2}{l}{$\widehat{\varrho}_{VX}^{\star QML}-\widehat{\varrho}_{VX}^{\star MLE}$} & \multicolumn{2}{l}{$\widehat{\varrho}_{VX}^{\star G (4)}-\widehat{\varrho}_{VX}^{\star MLE}$} & \multicolumn{2}{l}{$\widehat{\varrho}_{VX}^{\star CF(2)}-\widehat{\varrho}_{VX}^{\star MLE}$}\tabularnewline
 & & \multicolumn{1}{c}{Mean} & \multicolumn{1}{c}{SD} & \multicolumn{1}{c}{Mean} & \multicolumn{1}{c}{SD} & \multicolumn{1}{c}{Mean} & \multicolumn{1}{c}{SD} & \multicolumn{1}{c}{Mean} & \multicolumn{1}{c}{SD} & \multicolumn{1}{c}{Mean} & \multicolumn{1}{c}{SD}  \tabularnewline
\hline
$\kappa$ & 1 & 0.225534683391003
  & 0.360906607078388
 & 0.00573075432525952
 &  0.1234558966467
& 0.00687893425605536
 &  0.131295140589824
& 0.00718020761245674
& 0.117946015583042
& -0.0350442726643599
& 0.433586908228025
 \tabularnewline

$\kappa \theta _V$ & 0.04 & 0.00934982076124568
 & 0.0136919398983686
& 0.0000811776816608995
 & 0.00244132704919124
& -0.0000752375432525953
 & 0.00427299652924237
& 0.000159856920415224
& 0.0025351734655192
& -0.00301644152249135
& 0.0092182171458277
\tabularnewline

$\sigma$ & 0.2 & 0.000922141868512093
 &0.00481793846447103
 & -0.00024963321799308
 & 0.00139823074156555
& -0.0000841141868512104
& 0.00154207181218661
& -0.0000159325259515564
& 0.00139003771775847
& 0.00348332698961938
& 0.0100159996393291
\tabularnewline

$\kappa \theta _X$ & 0.03 & -0.0151026889370242
 & 0.0448129119713036
 & 0.000563081137197231
 & 0.0120824769235282
& 0.00130579395178201
 & 0.0170556062064675
& 0.000130125219031142
& 0.0128697451701373
& 0.0100674358349481
& 0.0390403445202174
\tabularnewline

$\rho$ & -0.8 & -0.00157714705882346
  & 0.0146157216802012
 & 0.00149111245674741
 & 0.00490283446432353
& 0.0000683737024221428
 & 0.00501619421620587
& -0.0000517750865051894
& 0.00493646738507949
& 0.00184755882352941
& 0.0255258412877958
\tabularnewline
\hline \hline
\end{tabular}}
\caption{\label{tab:heston}\textbf{Heston Asymptotic Assessment: } Panel \subref{tab:hestonbiasmore} displays bias and RMSE of the MLE estimator and the approximated MLE estimators. Panel \subref{tab:hestonbias} displays mean and standard deviation of ML estimation bias as well as the first two moments of the difference between the MLE estimator and the approximated MLE estimators.  Computed over a sample of 1,000 datasets, all of which generated as exact draws from the Heston model using the technology from \citet{broadiekaya06}. For a given dataset only parameter estimates were taken into consideration where all five estimators converged. Out of 1,000 this left 578 samples. The number of estimation successes is reported in Table \ref{tab:estimationsuccess} above. Approximate estimators are obtained through BG(4), a fourth order expansion using a Bilateral Gamma weight for the log stock variable and a Gamma weight for the variance variable, G(4), a fourth order expansion using a Gaussian weight for the log stock variable  and a Gamma weight for the variance variable, QML denotes a Gaussian approximation using the true conditional moments up to order 2, and CF(2) denotes the second-order likelihood expansions from \citet{aitsahalia08}.   }
\end{center}
\end{small}
\end{table}
\end{landscape}

\begin{figure}
\begin{center}
 \subfloat[$\kappa_V  $\label{fig:hestk}]{\scalebox{0.55}{\input{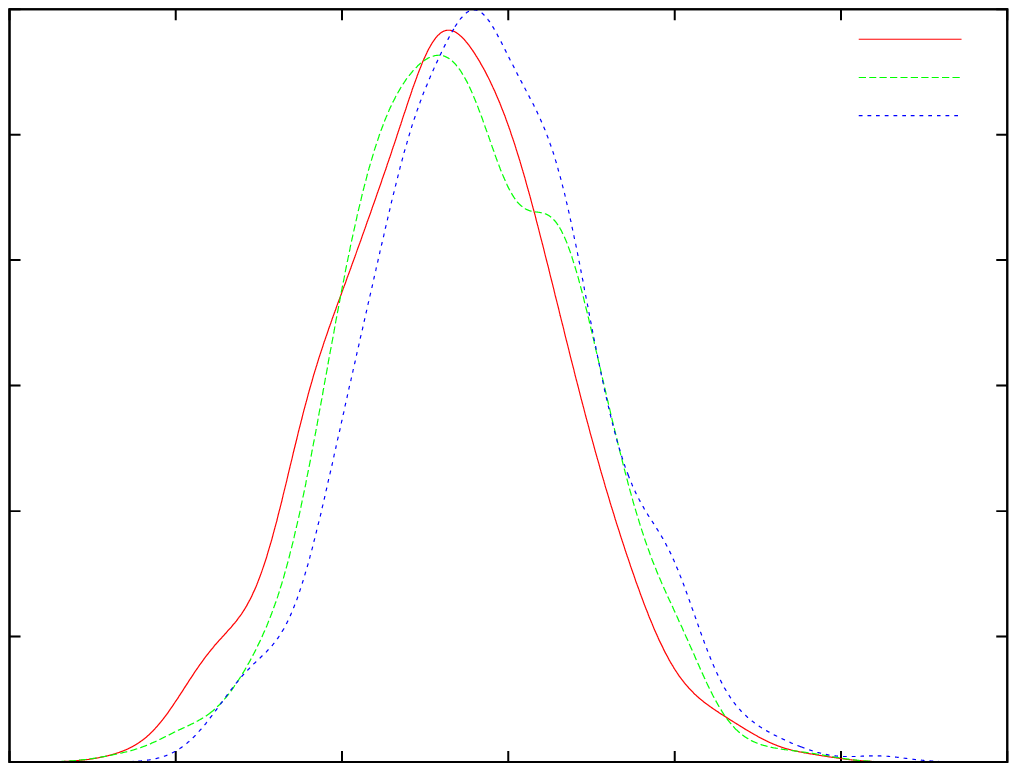}}}
 \subfloat[$\kappa \theta_V $\label{fig:hestkth}]{\scalebox{0.55}{\input{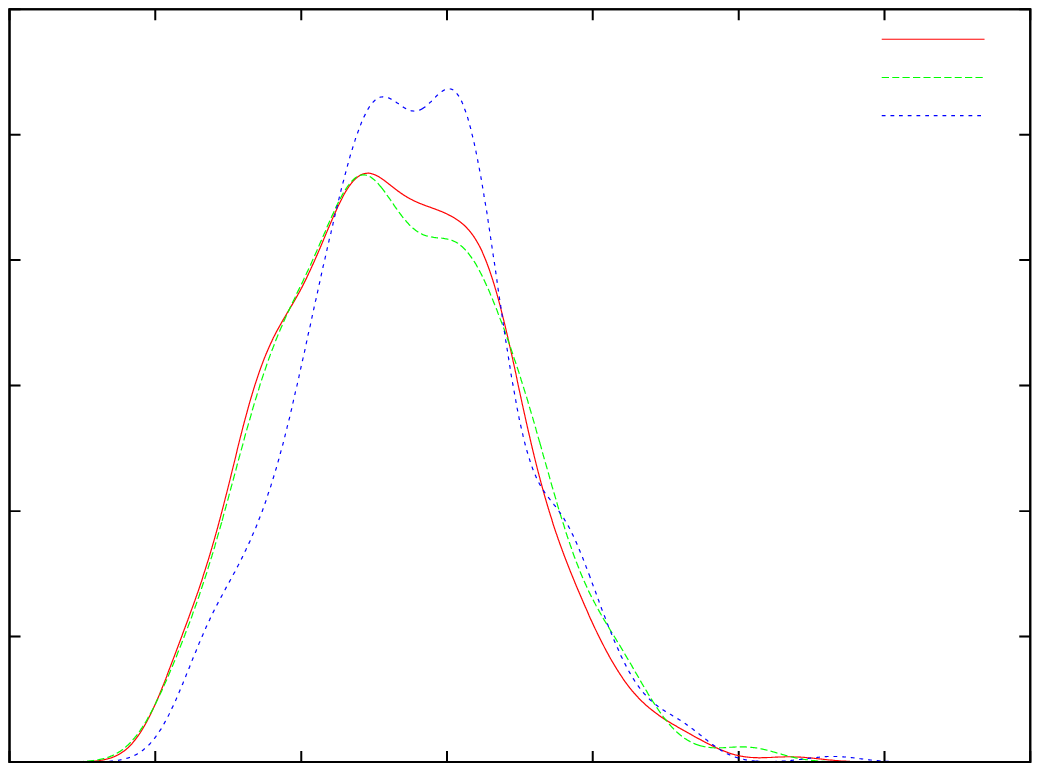}}} \\
 \subfloat[$\sigma $\label{fig:hestsig}]{\scalebox{0.55}{\input{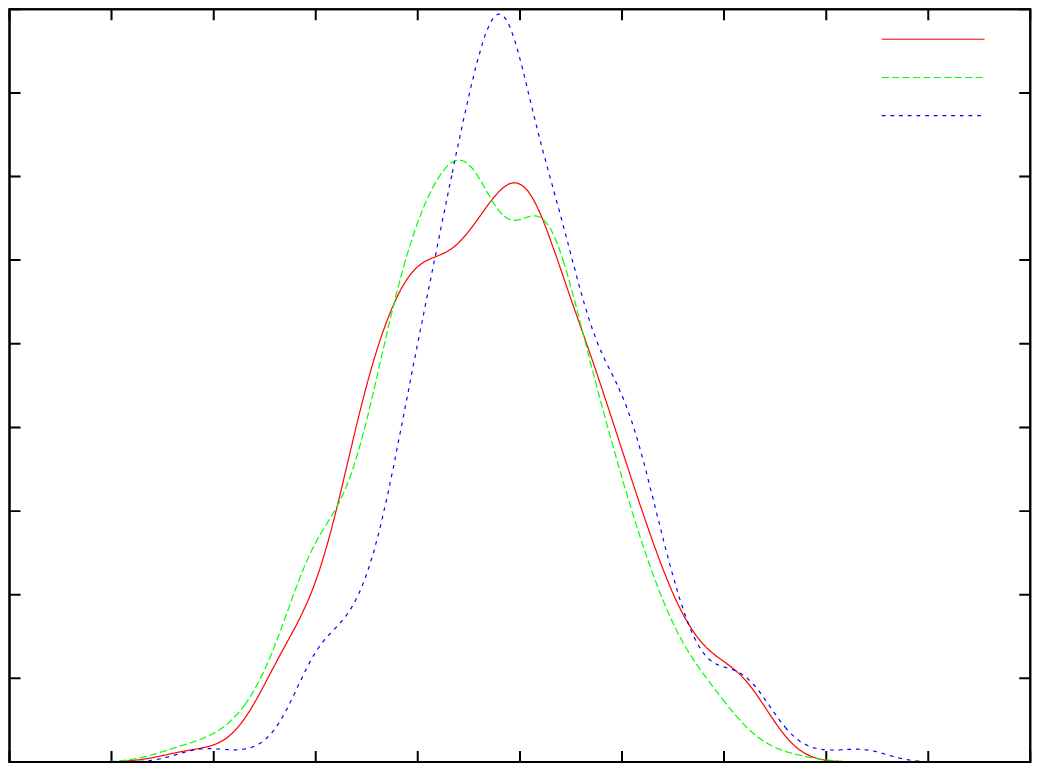}}} \subfloat[$\kappa \theta _X $\label{fig:hestkth2}]{\scalebox{0.55}{\input{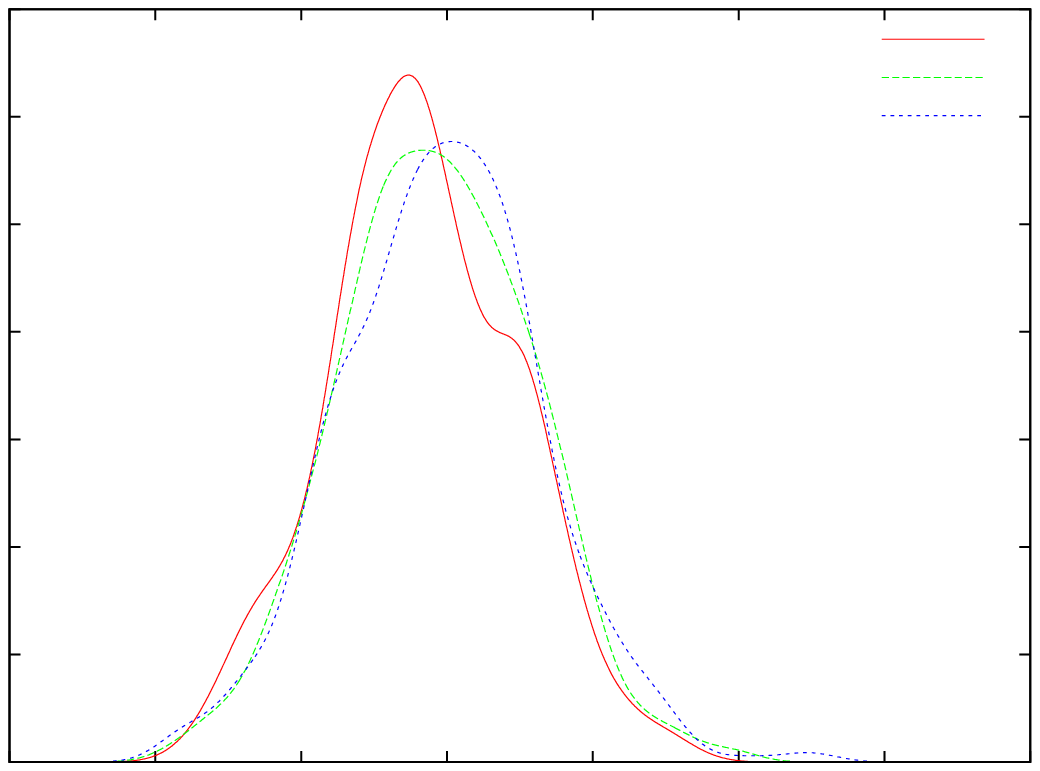}}}\\ \subfloat[$\rho \mid $\label{fig:hestrho}]{\scalebox{0.55}{\input{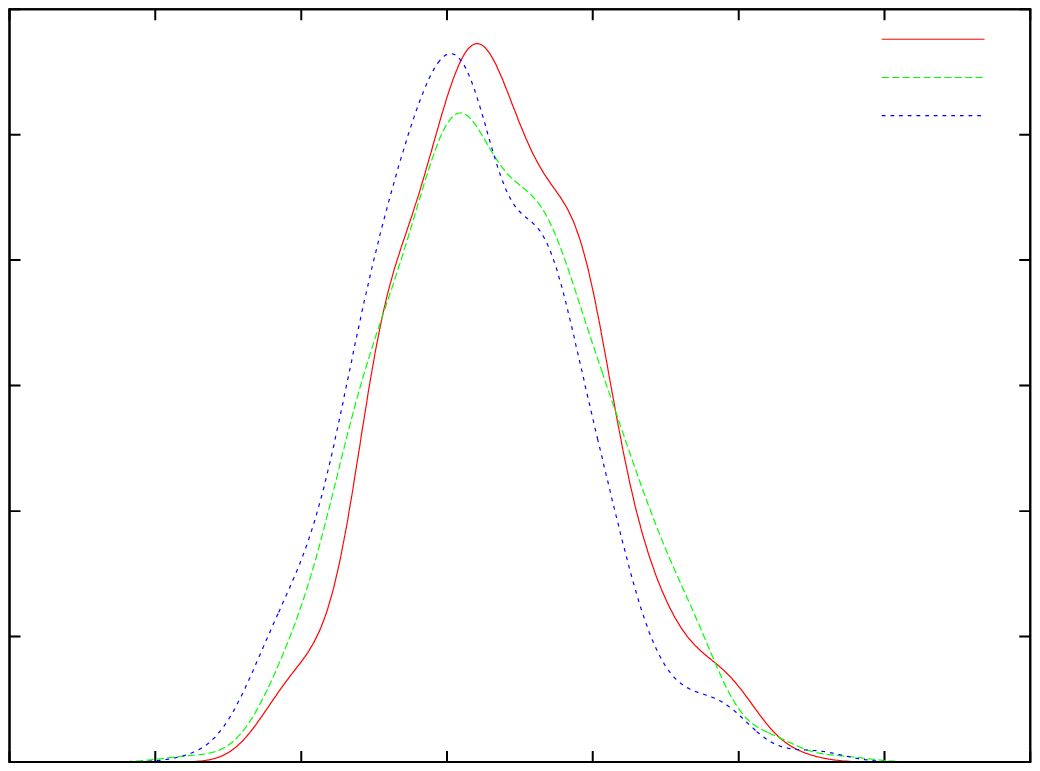}}} \caption{\label{fig:hestparms}\textbf{Posterior Densities for Heston's model: } The figure displays the marginal posterior distributions of the parameters of Heston's model conditional on the data. Bayesian estimation is performed using prior specification \eqref{eq:hestonprior} with the true transition density $g_{VX}$ obtained through Fourier inversion, closed-form density up to second ($g_{VX}^{(2)}$), and fourth  order ($g_{VX}^{(4)}$).}
\end{center}
\end{figure}

\clearpage
\end{appendix}

\nocite{AitSahaliaHansen}
\bibliographystyle{ecta}
\bibliography{../../masterbib/trunk/master}

\end{document}